\documentclass[
	,12pt%
	,pagesize%
	,headings=small%
	,paper=a4%
	,parskip=false%
	,abstract=on
	,toc=bibliography
]{scrartcl}
\addtokomafont{sectioning}{\normalfont\bfseries}
\setkomafont{title}{\normalfont}
\setkomafont{subtitle}{\normalfont}
\setkomafont{author}{\normalfont}
\setkomafont{date}{\normalfont}
\addtokomafont{pageheadfoot}{\scshape\small}
\setkomafont{caption}{\footnotesize}
\setkomafont{captionlabel}{\usekomafont{caption}\bfseries}
\setcapindent{0pt}
\usepackage{etoolbox}
\AtBeginEnvironment{abstract}{\footnotesize}

\usepackage{xpatch}
\makeatletter
\patchcmd{\@maketitle}{\huge}{\Large}{}{}
\makeatother

\usepackage[utf8]{inputenc}
\usepackage[T1]{fontenc}
\usepackage{csquotes}
\usepackage[english]{babel}
\usepackage{graphicx}
\graphicspath{{Pictures/}{/}}
\usepackage[export]{adjustbox}

\usepackage[usenames,dvipsnames]{xcolor}
\usepackage{tikz}
\usepackage{tikz-cd}
\usetikzlibrary{backgrounds}
\usetikzlibrary{angles}

\usepackage[draft]{fixme}
\usepackage[nointlimits]{amsmath}
\usepackage{amssymb,mathrsfs,mathtools,amsthm}
\usepackage{txfonts}
\usepackage{upgreek}
\usepackage{braket}
\usepackage{cancel}
\usepackage[expansion=true,protrusion=true]{microtype}
\usepackage{xkeyval}
\usepackage{overpic}
\usepackage{cite}
\usepackage{authblk}

\usepackage[final,%
	pdftex,%
	bookmarks,%
	bookmarksdepth=3,%
	breaklinks=true,%
	colorlinks=true,%
	urlcolor=NavyBlue,%
	linkcolor=NavyBlue,%
	citecolor=ForestGreen,%
]{hyperref}%
\usepackage[all]{hypcap}
\usepackage[capitalise,nameinlink]{cleveref}
\usepackage{quiver}

\renewcommand{\autoref}[1]{\cref{#1}}

\newtheorem{theorem}{Theorem} 
\newtheorem{lemma}[theorem]{Lemma}
\newtheorem{corollary}[theorem]{Corollary}
\newtheorem{proposition}[theorem]{Proposition}

\theoremstyle{definition}
\newtheorem{definition}[theorem]{Definition}

\theoremstyle{remark}

\addto\extrasenglish{%
}

\definecolor{RoyalBlue}{RGB}{0,32,96}
\definecolor{ForestGreen}{RGB}{34,139,34}
\definecolor{BrickRed}{RGB}{203, 65, 84}

\newcommand{\closure}{\mathrm{cl}}
\newcommand{\opening}{\mathrm{op}}
\newcommand{\cl}{\closure}
\newcommand{\op}{\opening}
\newcommand{\hop}{\widetilde{\op}}

\newcommand{\Arm}{{\mathrm{Arm}}}

\newcommand{\ArmProb}{P_{\varrho}}

\newcommand{\Pol}{{\mathrm{Pol}}}
\newcommand{\PolLong}{\Pol_d(n;r)}
\newcommand{\PolQuot}{\widehat{\Pol}}

\newcommand{\PolProb}{P_{\varrho,r}}
\newcommand{\PolQuotProb}{\hat{P}_{\varrho,r}}
\newcommand{\PolSympRedProb}{\hat{P}_{\sqrt{r},r}}
\newcommand{\PolQuotVol}{\widehat{\vol}_{\varrho,r}}

\newcommand{\QuotMap}{\pi}

\newcommand{\diag}{\mathrm{diag}}
\newcommand{\vect}{\mathrm{vec}}

\newcommand{\SO}{\mathrm{SO}}

\newcommand{\Shift}{\sigma}

\newcommand{\hShift}{\widetilde{\Shift}}

\DeclareMathAlphabet{\mathcal}{OMS}{cmsy}{m}{n}

\newcommand{\Sphere}{\mathbb{S}}
\newcommand{\Disk}{\mathbb{B}}

\DeclareDocumentCommand{\Hess}{ O{} }{\operatorname{Hess}_{#1}}

\newcommand{\qand}{\quad \text{and} \quad}

\newcommand{\transp}{^{\mathsf{T\!}}}
\newcommand{\adj}{^{*\!}}

\newcommand*{\dd}{\mathop{}\!\mathrm{d}}

\newcommand{\ceq}{\coloneqq}

\newcommand{\R}{{\mathbb{R}}}

\newcommand{\abs}[1]{\mathopen{}\left\lvert#1\right\rvert\mathclose{}} 
\newcommand{\nabs}[1]{\lvert{#1}\rvert} 

\newcommand{\innerprod}[1]{\mathopen{}\left\langle#1\right\rangle\mathclose{}}
\newcommand{\ninnerprod}[1]{\langle{#1}\rangle}

\DeclarePairedDelimiter\pars{\lparen}{\rparen}

\newcommand{\intervalcc}[1]{\mathopen{}\left[#1\right]\mathclose{}}

\DeclareMathOperator{\tr}{tr}

\newcommand{\vol}{{\mathrm{vol}}}
\newcommand{\dvol}{\mathrm{d}\vol}

\newcommand{\push}{\#}

\begin{document}
\title{CoBarS: Fast reweighted sampling \\ for polygon spaces in any dimension}
\author[1]{Jason Cantarella}
\author[1,2]{Henrik Schumacher}

\affil[1]{Mathematics Department, University of Georgia, Athens, GA, USA}
\affil[2]{Faculty of Mathematics, Chemnitz University of Technology, Chemnitz, Germany}

\maketitle

\begin{abstract}
\begin{small}
We present the first algorithm for sampling random configurations of closed $n$-gons with any fixed edgelengths 
$r_1, \dots, r_n$ in any dimension $d$ which is proved to sample correctly from standard probability measures on these spaces. We generate open $n$-gons as weighted sets of edge vectors on the unit sphere and close them by taking a M\"obius transformation of the sphere which moves the center of mass of the edges to the origin. Using previous results of the authors, such a M\"obius transformation can be found in $O(n)$ time. The resulting closed polygons are distributed according to a pushforward measure. The main contribution of the present paper is the explicit calculation of reweighting factors which transform this pushforward measure to any one of a family of standard measures on closed polygon space, including the symplectic volume for polygons in $\R^3$. For fixed dimension, these reweighting factors may be computed in $O(n)$ time. Experimental results show that our algorithm is efficient and accurate in practice, and an open-source reference implementation is provided.
\\

\noindent
\textbf{MSC-2020 classification:} 
60D05, 
65D18, 
82D60 
\end{small}
\end{abstract}


\section{Introduction}\label{sec:intro}

We consider configurations of $n$ points in $\R^d$ with positions $v_1, \dotsc, v_n$ separated by a length vector $r$ of fixed distances $r_1, \dotsc, r_n >0$. We treat indices cyclically, so we may write the displacement vectors $v_{i+1} - v_i = r_i \, y_i$ where $y_i$ is a unit vector in $\R^d$, letting $r_n \, y_n = v_n - v_1$. The elements of such a space are polygons with edgelengths given by the $r_i$. Our goal in this paper is to give an efficient way to randomly sample such polygons. 

This sampling problem is of interest in the statistical physics of polymers, where the configuration space of polygons with fixed edgelengths is the state space for the~\emph{freely-jointed chain model}~\cite{Kramers1946} of a ring polymer. (See the survey paper~\cite{Orlandini2007} for many applications of these kinds of models in physics and biology.) The same space is studied in robotics~\cite{Farber:2008gg}, where it models the kinematic configuration space of a robot arm with spherical revolute joints forming a closed loop.\footnote{If this seems an unusual special case, the standard example~\cite{MilgramTrinkle2004} of a kinematic loop is that of a robot in a fixed position manipulating an object which is also constrained, such as a door handle. Here there are implicit constraints connecting the base of the robot to the door's hinge and the hinge to the door handle. In more complicated situations, multiple kinematic loops may be present at the same time, but this is beyond the scope of the present paper.} It is also of intrinsic interest in differential geometry and topology~\cite{MillsonKapovich1996, FarberFromm2013, Mandini2014, Gallet2017}.

Various algorithms have been proposed to construct random polygons~\cite{CSS,CDSU,CantarellaShonkwilerAAP, Millett1994,Vologodskii:1979ik,Klenin:1988dt,AlvaradoCalvoMillett2011,Moore2005,Moore2004,Varela2009}. However, all of them suffer from one or more deficiencies; they are either explicitly restricted to dimension $d=2$ or $d=3$, not proved to sample the correct measure, and/or only generate equilateral polygons. Therefore, there is a need for a polygon sampling method which is fast, can be proved to sample the correct measure, and gracefully accommodates arbitrary choices of dimension and edgelengths. In this paper, we present such a method: conformal barycenter sampling (CoBarS).

We start by observing that it is easy to construct configurations of $n+1$ points $v_1, \dots, v_{n+1}$ so that $\nabs{v_{i+1} - v_{i}} = r_i$ by sampling unit vectors $x_i$ uniformly from a product of spheres and letting $v_i = \sum_{j<i} r_i \, x_i$. 
These configurations have the correct edgelengths, but they usually fail to close because $v_{n+1} = v_1$ is satisfied if and only if $\sum_{i=1}^n r_i \, x_i = 0$. 
However, we will build a map (Definition~\ref{defn:shift and conformal barycenter closure}) from the space of open polygons to the space of closed polygons using the conformal barycenter (Definition~\ref{defn:director and conformal barycenter}). 
This map is only defined implicitly, but can be computed efficiently using the algorithm in~\cite{CantarellaSchumacher2022}. This gives us a fast sampling algorithm, but the resulting samples are biased. The main theoretical contribution of this paper is a fast and explicit way to compute reweighting factors which eliminate this sampling bias. We then see in experiments that we can compute integrals over the configuration space of polygons quickly and accurately using this reweighting. The resulting method is faster and more general than the Action-Angle method described in~\cite{CSS,CDSU}.

We note that the idea of generating closed polygons from open ones or (more or less equivalently) polygons with given edgelengths from arbitrary closed polygons is definitely not a new one and a variety of polygon closure or resampling algorithms have been proposed~\cite{CantarellaChapmanReiterShonkwiler2018}. Any of these could be used to generate (biased) samples from closed polygon space, which could in principle be reweighted to sample the standard measure. The key new feature of this approach is that our closure algorithm is mathematically controlled enough that we can prove that it (almost) always converges, provide time bounds, and explicitly compute reweighting factors.

\section*{Acknowledgments}

The authors are grateful to many colleagues for helpful discussions of polygons and hyperbolic geometry, especially Clayton Shonkwiler, Tetsuo Deguchi, and Erica Uehara. 


\section{Opening and Closing Polygons}
\label{sec:opening and closing polygons}

\begin{figure}
\begin{center}
	\capstart 
\newcommand{\inc}[2]{\begin{tikzpicture}
    \node[inner sep=0pt] (fig) at (0,0) {\includegraphics{#1}};
	\node[above right= 0ex] at (fig.south west) {\begin{footnotesize}(#2)\end{footnotesize}};    
\end{tikzpicture}}%
\presetkeys{Gin}{
	trim = 0 100 15 20, 
	clip = true,  
	width = 0.333\textwidth
}{}
	\hfill
	\inc{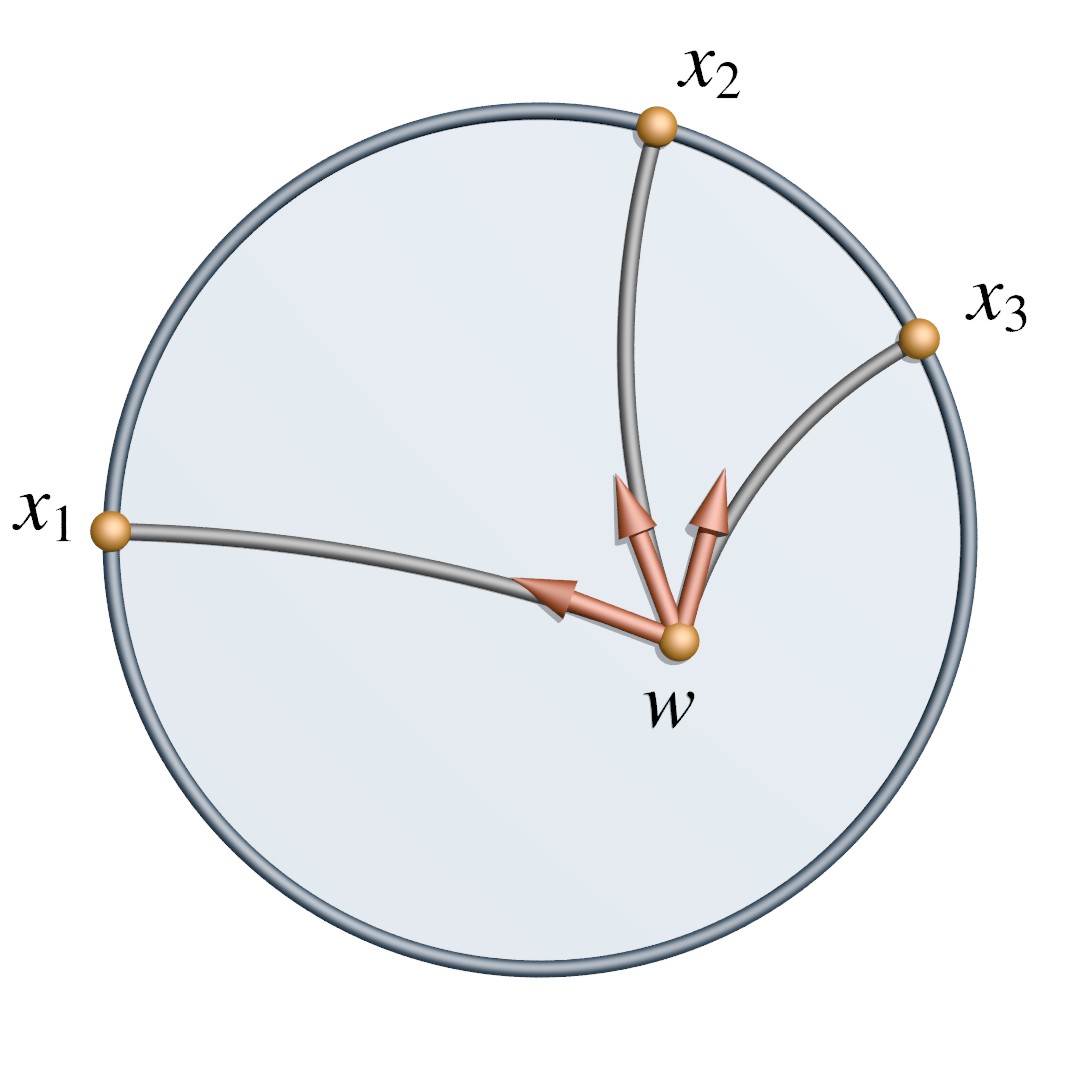}{a}%
	\hfill
	\inc{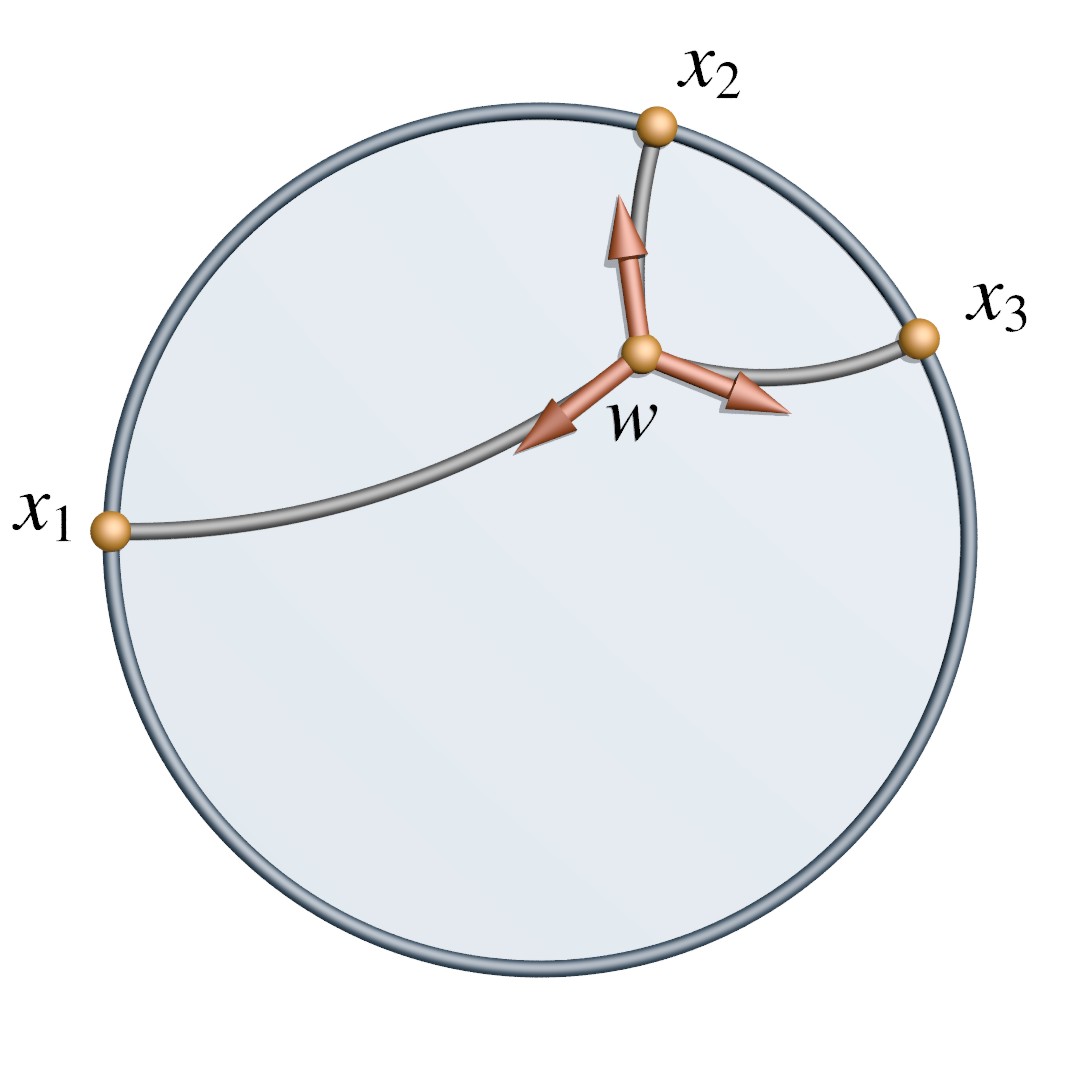}{b}%
	\hfill	
	\inc{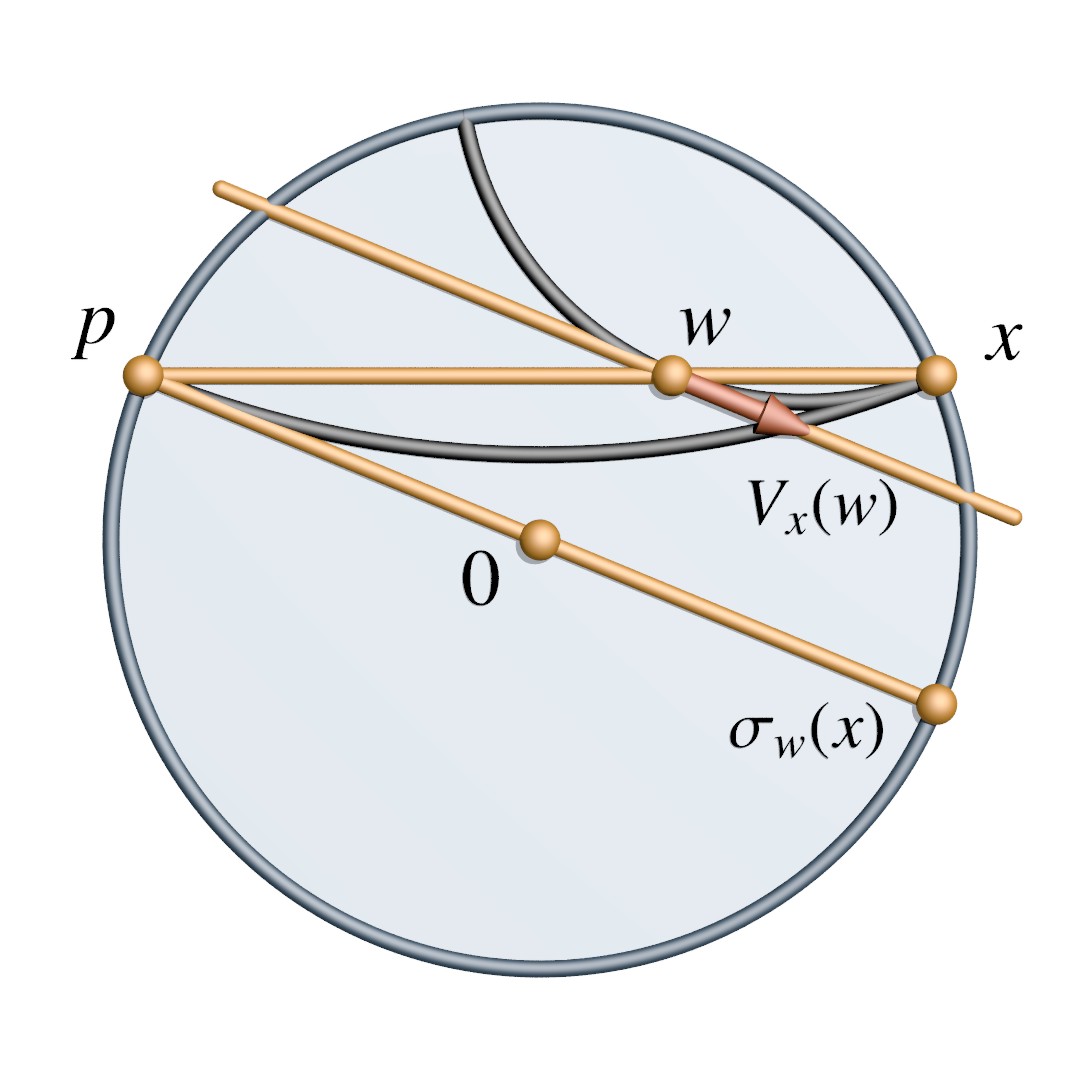}{c}%
	\hfill{}	
	\caption{%
	(a)	
	Geodesics joining some point $w$ in $\Disk$ to three points $x_1, x_2, x_3$ in~$\Sphere$ and the corresponding conformal directors.
	The directors $V_{x_i}(w)$ do not sum up to $0$.
	(b)
	Same as (a), but here the sum of directors $V_{x_i}(w)$ vanishes;  thus $w$ is the conformal barycenter of the $x_i$.
	(c) Geometric construction of the directors: 
	Each geodesic emanating from $x$ intersects the secant $p x$ in the same angle. Thus
	all directors $V_x(w)$ for $w$ on the secant $p x$  point in the same direction.
	}
\label{fig:conformalbary}
\end{center}
\end{figure}
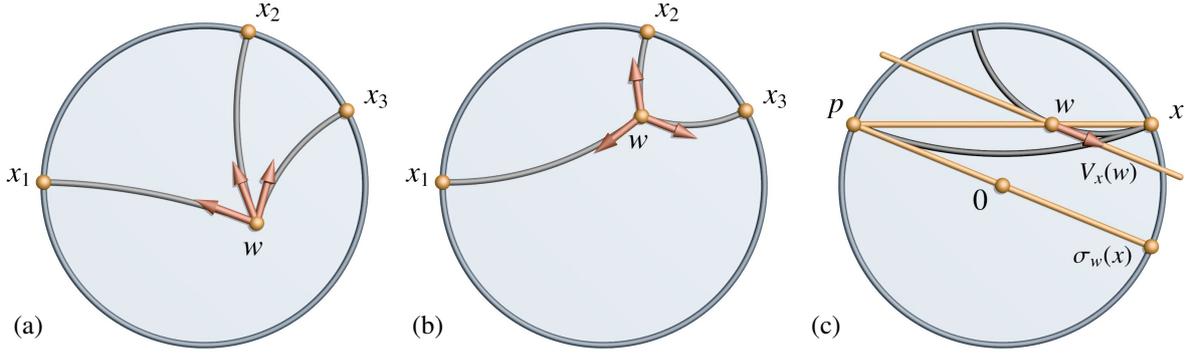

In order to sample, we need to carefully define a probability space and a corresponding measure.  
Set $\Sphere \ceq S^{d-1}$, $\Disk \ceq B^{d-1} = \set{ x \in \R^d \mid \nabs{x} < 1}$, and $\bar\Disk \ceq \set{ x \in \R^d \mid \nabs{x} \leq 1 }$ and assume that~$r \in \R^n_+$. Then we can define spaces of open polygonal ``arms'' and closed polygons by letting 
\[
\Arm_d(n) \ceq \Set{ x \in (\Sphere)^n \vphantom{\textstyle \sum_{i = 1}^n}} \qand \PolLong \ceq \Set{ y \in \Arm_d(n) | \textstyle \sum_{i = 1}^n r_i \, y_i = 0}
\]
and by associating directions $x_i$ or $y_i$ with vertices $v_i$ by adding and shifting the center of mass of the $v_i$ to the origin. We make the standing assumption that $d \geq 2$ and $n \geq 3$. Since polygons cannot close if some edgelength $r_j$ is greater than or equal to $\frac{1}{2} \sum_{i=1}^n r_i$, we will also make the standing assumption that
\begin{align}
\label{eq:stable_r}
\text{for each $j \in \set{1, \dotsc, n}$:} \quad  r_j < \frac{1}{2} \sum_{i=1}^n r_i.
\end{align}
We will assume that $d$, $n$ and $r$ are fixed, and replace $\Arm_d(n)$ with $\Arm$ and $\PolLong$ with $\Pol$ for brevity.

We define a Riemannian metric $g_\varrho$ by scaling the standard metric $g_\Sphere$ on each $\Sphere$ by $\varrho_i^2$ and assuming that the tangent spaces to different spheres are orthogonal: 
\[
	g_\varrho(u,v) \ceq \sum_{i = 1}^{n} \varrho_{i}^{2} \, g_{\Sphere}(u_{i},v_{i})
	\quad 
	\text{for all $u,v \in T_x \Arm = T_{x_1} \Sphere \times \dotsm \times T_{x_n} \Sphere$.}
\] 
This is the restriction of the metric $\ninnerprod{u,v}_\varrho = \sum_{i=1}^n \varrho_{i}^2 \ninnerprod{u_i,v_i}$ on $T (\R^d)^n$.

The metric $g_\varrho$ generates a corresponding volume measure $\vol_\varrho$ and probability measure $\ArmProb$ on $\Arm$, which happen to be products of measures on $\Sphere$: 
\[
	\quad 
	\vol_\varrho = \prod_{i=1}^n \varrho_i^{d-1} \vol_\Sphere,
	\quad 
	\ArmProb \ceq \frac{\vol_\varrho}{\vol_\varrho(\Arm)} 
	= 
	\prod_{i=1}^n \frac{\vol_\Sphere}{\vol_\Sphere(\Sphere)}.
\]
We note that $\ArmProb$ is independent of $\varrho$, even though $\vol_\varrho$ and $g_\varrho$ depend on our choice of $\varrho_i$. When $\varrho_i = \sqrt{r_i}$,
the volume corresponds to the symplectic volume of Millson and Kapovich~\cite{MillsonKapovich1996}. When $\varrho_i = r_i$, the volume is equivalent to taking a product of spheres with radii $r_i$. 
 
It is known that $\Pol$ is a Riemannian submanifold of the Riemannian manifold $\Arm$, with isolated singularities only at points where all $y_i$ are colinear~\cite[Prop. 3.1]{FarberFromm2013}. So $\Pol$ inherits a submanifold metric $g_{\varrho,r}$, (Hausdorff) volume measure $\vol_{\varrho,r}$, and probability measure $\PolProb$. We will assume for the rest of the paper that we have fixed $\varrho$ and $r$, and therefore fixed $\PolProb$, which we will refer to as~\emph{the} measure on $\Pol$.  

We will now connect polygon spaces to hyperbolic geometry (as in~\cite{MillsonKapovich1996}). We think of $\Disk$ as the Poincar\'e disk model of hyperbolic space, where $\Sphere$ is the sphere at infinity. 

\begin{definition} Given $x_i \in \Sphere$, at every $w \in \Disk$ there is a unique geodesic ray joining $w$ to $x_i$. The unit tangent vector to this geodesic ray is called the \emph{director} $V_{x_i}(w)$. If $x \in \Arm$, $r \in \R^n_+$ and $w_*$ has the property that $\sum r_i \, V_{x_i}(w_*) = 0$, we say $w_*$ is a~\emph{conformal barycenter} of $x$ with weights $r$. (See \autoref{fig:conformalbary} for a $2$-dimensional illustration.)
\label{defn:director and conformal barycenter} 
\end{definition}

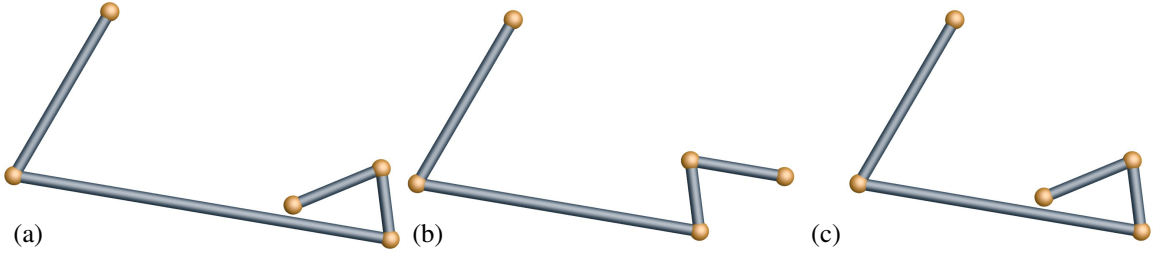
\begin{figure}
\begin{center}
	\capstart 
\newcommand{\inc}[2]{\begin{tikzpicture}
    \node[inner sep=0pt] (fig) at (0,0) {\includegraphics{#1}};
	\node[above right= 0ex] at (fig.south west) {\begin{footnotesize}(#2)\end{footnotesize}};    
\end{tikzpicture}}%
\presetkeys{Gin}{
	trim = 0 0 0 0, 
	clip = true,  
	width = 0.333\textwidth
}{}
	\hfill
	\inc{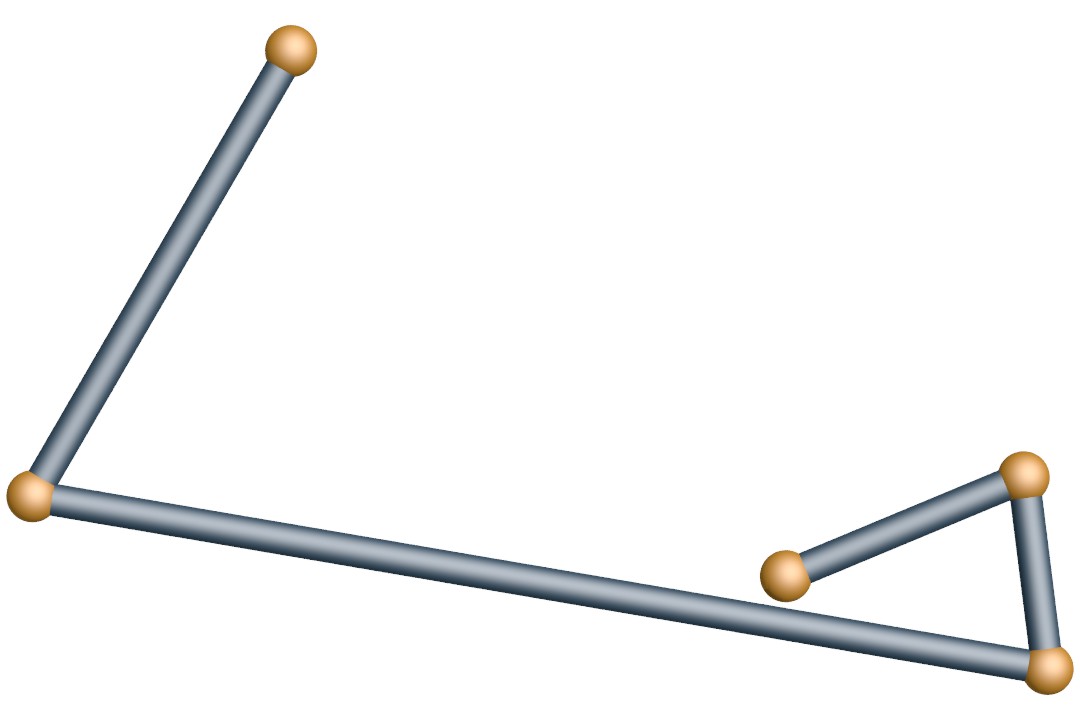}{a}%
	\hfill
	\inc{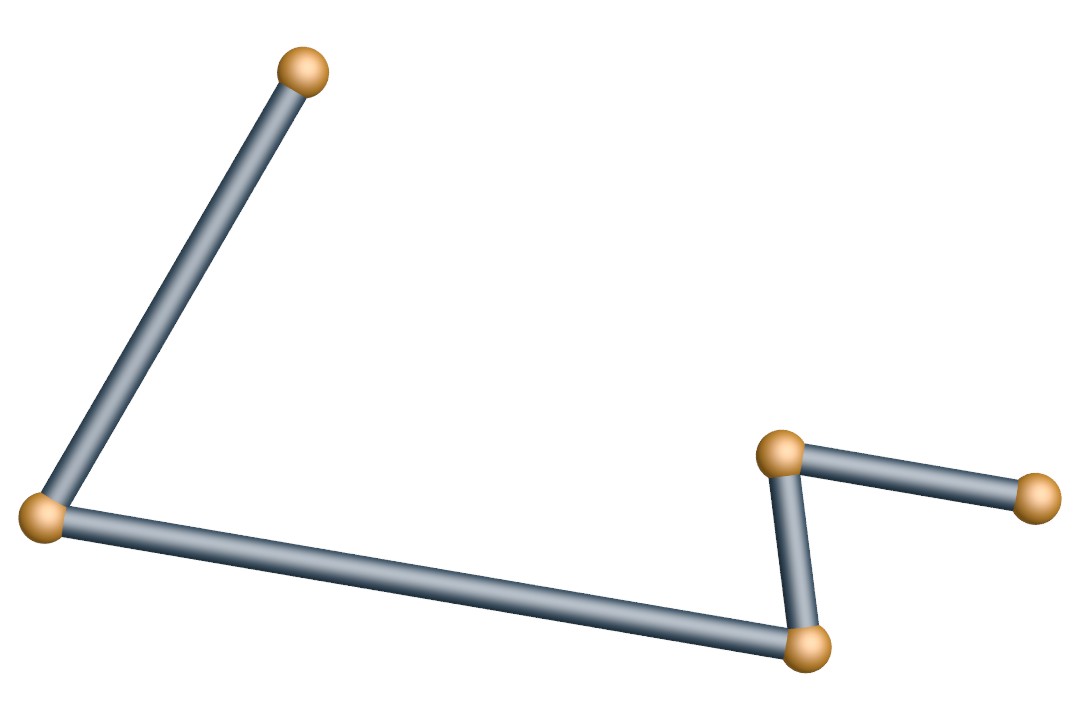}{b}%
	\hfill	
	\inc{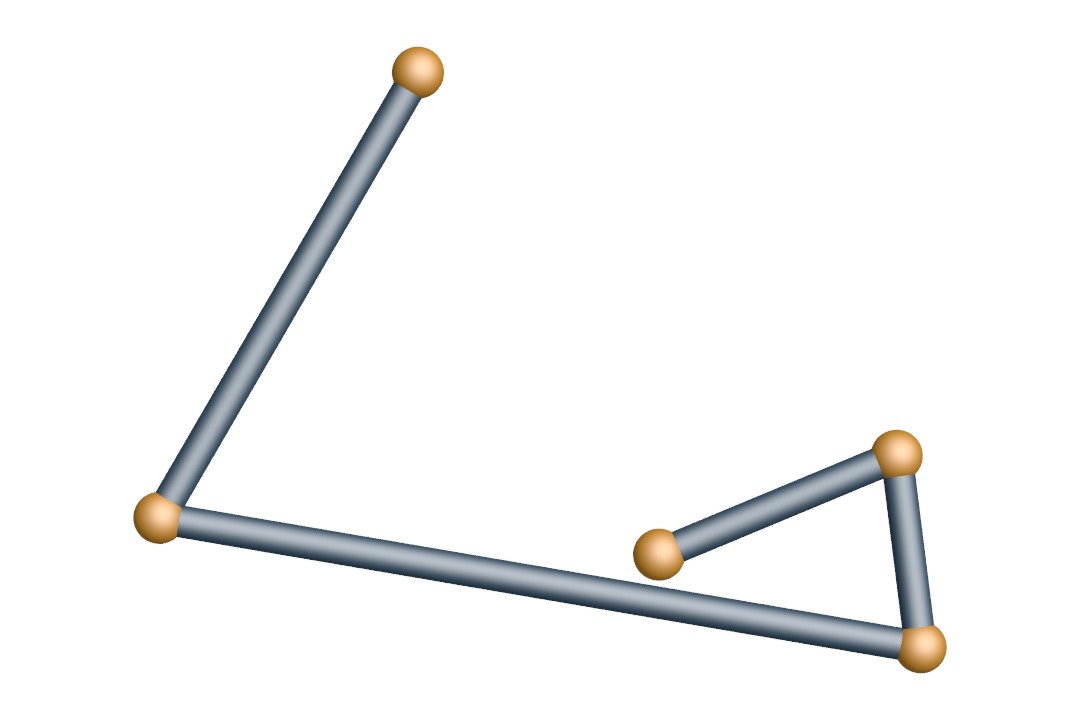}{c}%
	\hfill{}	
	\caption{%
	(a)	
	This polygon has edgelengths $r = (8, 16, 3, 4)$, which do not satisfy~\eqref{eq:stable_r}. 
	Regardless of their directions, the short edges cannot close the polygon, so we do not consider such $r$.
	(b)
	This polygon has edgelengths $r = (8, 12, 3, 4)$, which do satisfy~\eqref{eq:stable_r}. For this polygon $x_2 = x_4$ and their edgelength sum is greater than the remaining two edges. Thus $x$ is not stable with respect to $r$ in the sense of~\autoref{defn:stable weights and directions}. It will turn out to be the case that we cannot close this polygon. 
	(c) This polygon has edgelengths $r = (8, 12, 3, 4)$ and unique directions ($x_i \neq x_j$ when $i \neq j$). This $x$ is stable and we will be able to close this polygon.
	}
\label{fig:stable}
\end{center}
\end{figure}
Even under the assumption~\eqref{eq:stable_r}, there will be some polygons we are unable to close with our method (see \autoref{fig:stable}). So we now impose an additional (mild) hypothesis.
\begin{definition}
If $x \in \Arm$ and for every $v \in \Sphere$, $\sum_{i \, \mid\,  x_i = v} r_i < \frac{1}{2} \sum_{i=1}^n r_i$, then we say $x$ is stable (with respect to $r$). 
\label{defn:stable weights and directions}
\end{definition}

\begin{proposition}[\cite{MR0857678}, Section 11]
If $x \in \Arm$ is stable with respect to $r$, then it has a unique conformal barycenter $w_*(x)$ in the interior of $\Disk$.
\end{proposition}

By construction, the conformal barycenter is equivariant under hyperbolic isometries of $\Disk$, i.e. for every hyperbolic isometry $\varphi: \Disk \to \Disk$ we have 
\[
w_*(\varphi(x_1),\dotsc, \varphi(x_n)) = \varphi(w_*(x_1,\dotsc,x_n)),
\]
where we define $\varphi$ on $\Sphere$ to be the unique continuation of $\varphi$ on $\Disk$. 

Now the geodesics from $0$ to each $x \in \Sphere$ are radial lines. Because the Poincar\'e metric at the origin is twice the Euclidean metric we have $V_x(0) = \frac{1}{2} x$. Hence, if $w_*(x_1,\dotsc,x_n) = 0$, then $\sum_{i=1}^n r_i \, x_i = 2 \sum_{i=1}^n r_i \, V_x(0) = 0$ and the polygon is closed. 
By equivariance, we know that if $\varphi$ is a hyperbolic isometry that brings $w_*(x_1,\dotsc,x_n)$ to the origin, then $(\varphi(x_1), \dotsc, \varphi(x_n) )$ will be a closed polygon.
However, there are many such isometries. 
We now choose a specific one.

\begin{definition}\label{defn:shift and conformal barycenter closure}
For any $w \in \Disk$, there is a unique hyperbolic translation which maps $w$ to $0$. We call this a~\emph{shift map} and denote it by $\Shift(w,-)$. This map extends to the sphere at infinity and is given by the formula
\begin{align}
\label{eq:shift}
	\Shift \colon \Disk \times \bar \Disk \to \bar \Disk,
	\quad
	\Shift(w,z) 
	\ceq
	\frac{(1- \nabs{w}^2) \, z - (1 + \nabs{z}^2 - 2 \, \ninnerprod{w,z}) \, w}{1- 2 \ninnerprod{w,z} + \nabs{w}^2 \nabs{z}^2}.
\end{align}
For $x \in \Arm$, we define $\Shift(w,x) \in \Arm$ by $\Shift(w,x)_i = \Shift(w,x_i)$. If $x$ is stable with respect to~$r$, the~\emph{conformal barycenter closure} $y_*(x) \in \Pol$ is defined by~$y_*(x) \ceq \Shift(w_*(x),x)$. The~\emph{conformal closure map} $\cl \colon \Arm \rightarrow \Disk \times \Pol$ is given by 
\begin{align*}
	\cl(x) = \pars[\Big]{ w_*(x), y_*(x) }
	.
\end{align*}
\end{definition}

The conformal barycenter closure gives us a canonical way to close polygons without changing their edgelengths, as long as the initial directions are stable. We can even interpolate between the open polygon $x$ and the closed polygon $y$ by taking $\Shift(t \, w_*,x)$, $t \in \intervalcc{0,1}$, as in \autoref{fig:closure}. Since the conformal barycenter is only defined implicitly, we cannot give a closed-form expression for $\cl$. However, we can evaluate $\cl$
using the algorithm in~\cite{CantarellaSchumacher2022} in $O(n)$ time using the open-source implementation in~\cite{Cantarella_ConformalBarycenter}. Further, we have a closed form for the inverse of $\cl$, which we discuss next. 
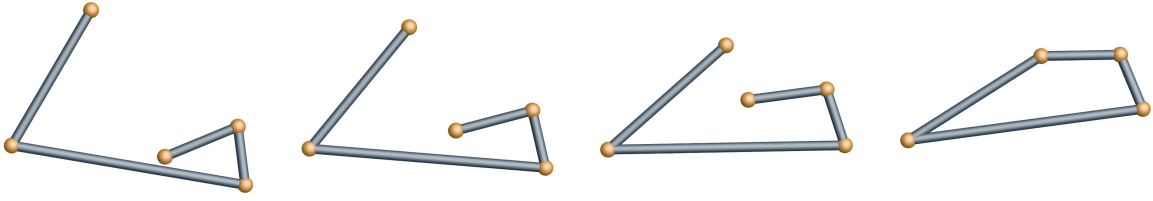
\begin{figure}
\begin{center}
	\capstart 
\newcommand{\inc}[2]{\begin{tikzpicture}
    \node[inner sep=0pt] (fig) at (0,0) {\includegraphics{#1}};
	\node[above right= 0ex] at (fig.south west) {\begin{footnotesize}#2\end{footnotesize}};    
\end{tikzpicture}}%
\presetkeys{Gin}{
	trim = 50 40 50 40, 
	clip = true,  
	width = 0.25\textwidth
}{}%
	\inc{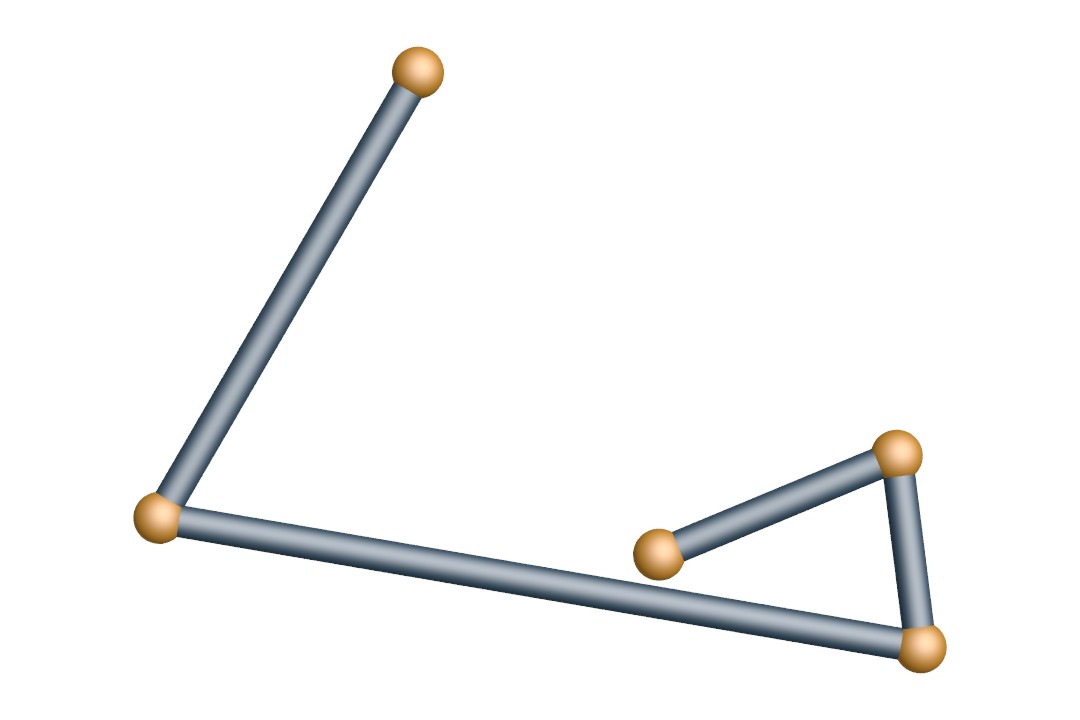}{}%
	\inc{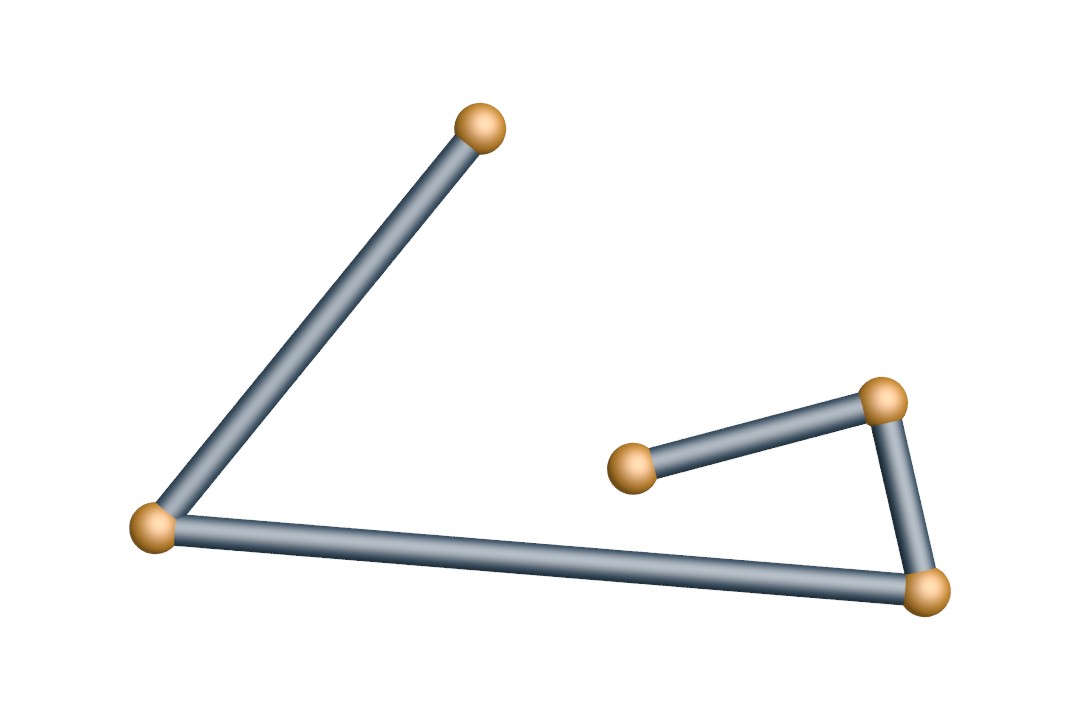}{}%
	\inc{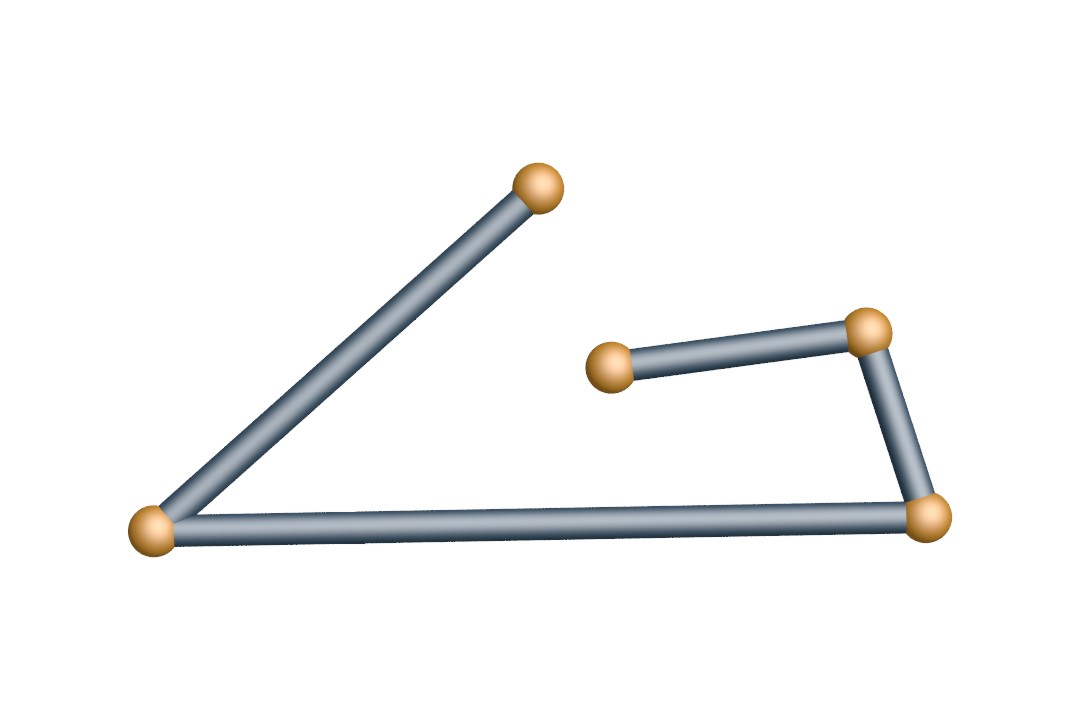}{}%
	\inc{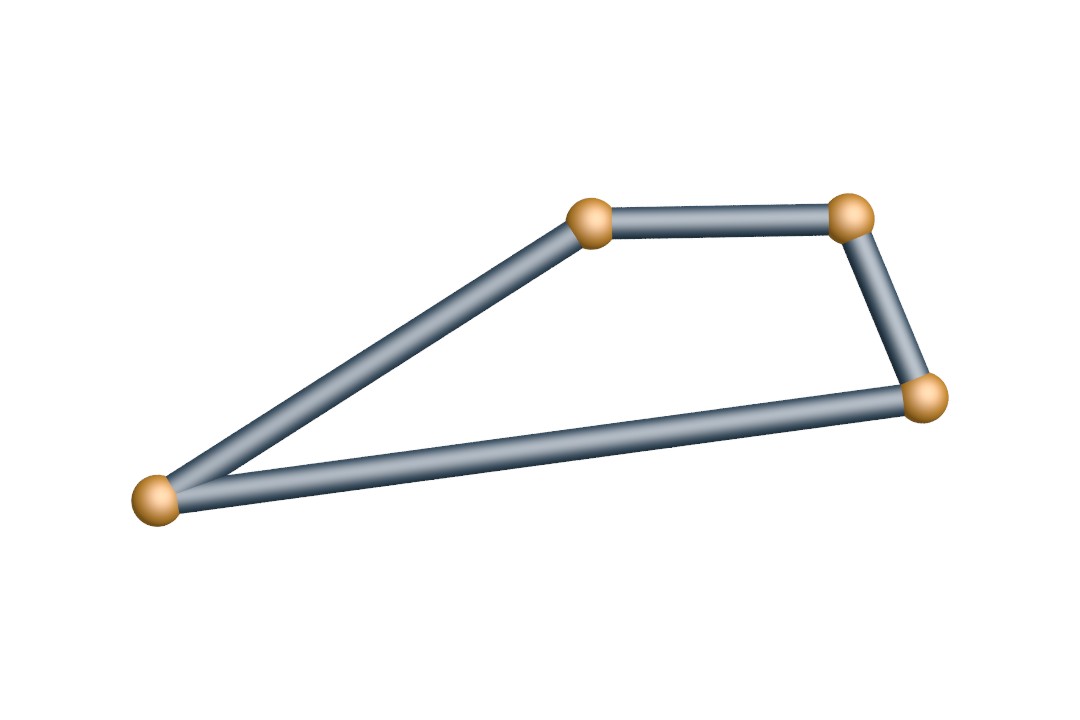}{}
	\caption{%
	From left to right, we interpolate between the original polygon $x$ and its conformal closure $y$ along the path $\Shift(t \, w_*(x),x)$. The edgelength vector $r = (4,6,3,2)$ remains the same throughout. This construction depends on the hypothesis that the initial $(x,r)$ are a stable pair; if not, we could not guarantee the existence of the conformal barycenter $w_*(x)$.}
\label{fig:closure}
\end{center}
\end{figure}
 
\begin{definition}
The~\emph{conformal opening map} $\op \colon \Disk \times \Pol \rightarrow \Arm$ is given by 
\[
\op(w,y) = \Shift(-w,y).
\]
\end{definition}

\begin{figure}[t]
\begin{center}
	\capstart
	\begin{tikzpicture}
	    \node[anchor=south west,inner sep=0] (image) at (0,0,0) {%
		    \includegraphics[width=0.6\textwidth,trim={0 100 0 100}, clip = true]{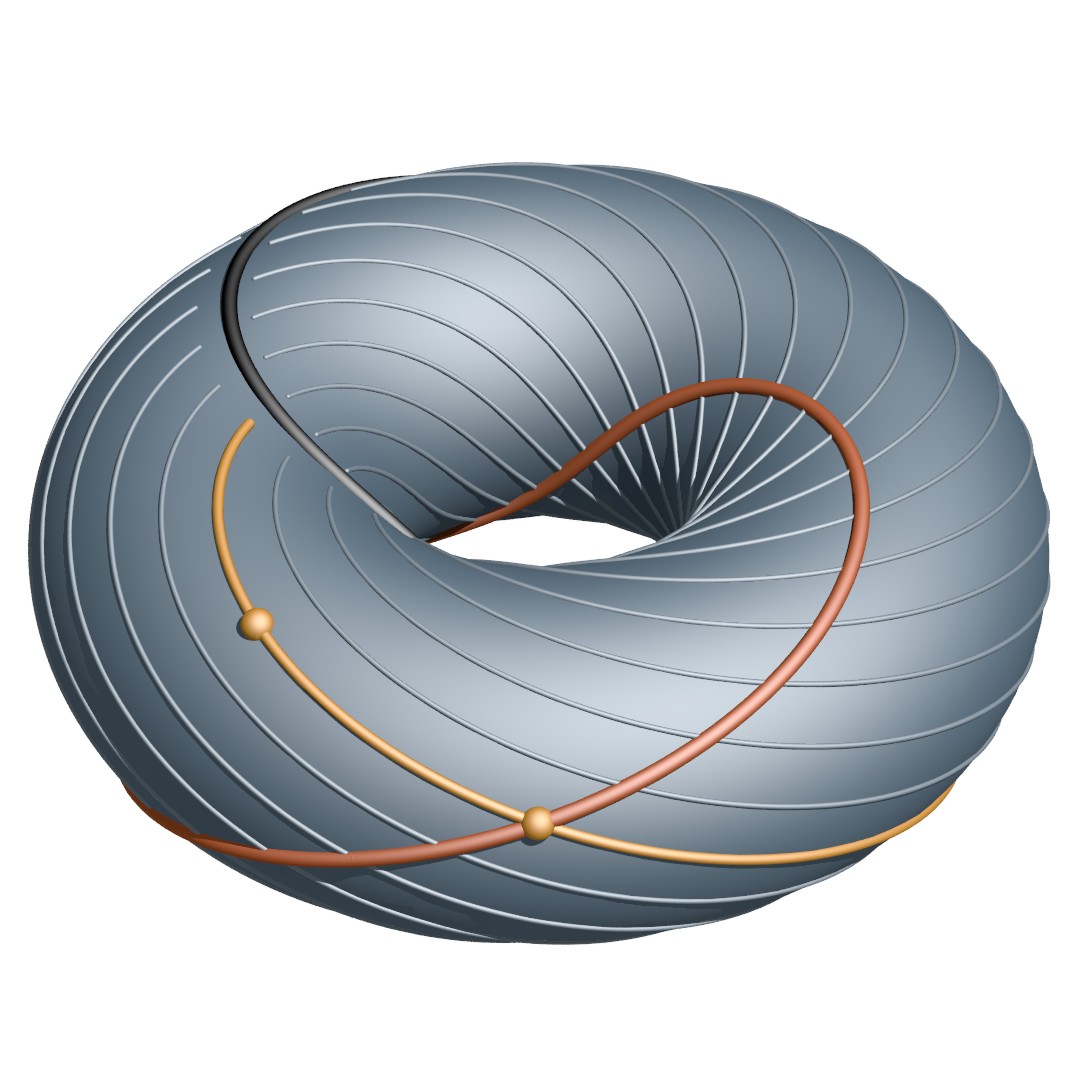}%
		};
	    \begin{scope}[x={(image.south east)},y={(image.north west)}]
			\node at (0.15 ,0.95,0) {$\Arm \setminus \Arm^\times$};
			\draw[line width=0.75pt] (0.16 ,0.92,0) -- (0.215,0.800,0);
			\node at (0.75 ,0.95,0) {$\Arm^\times$};
			\node at (0.95 ,0.85,0) {$\Pol^\times$};
			\draw[line width=0.75pt] (0.94 ,0.81,0) -- (0.8,0.55,0 );
			\node at (0.25 ,0.45,0) {$x$};
			\node at (0.485,0.225,0) {$y$};
			\node at (0.92 ,0.05,0) {$\Shift(\Disk,x)$};
			\draw[line width=0.75pt] (0.88 ,0.08,0) -- (0.795,0.155,0 );
	    \end{scope}
	\end{tikzpicture}
\end{center}
\caption{We will use the conformal barycenter construction to prove that $\Disk \times \Pol^\times$ is diffeomorphic to $\Arm^\times$ by the maps $\cl$ or $\op$~(\autoref{theo:FastJ}). Since $\Disk$ is contractible, this shows that $\Arm^\times$ deformation retracts onto $\Pol^\times$. The picture above shows this schematically: the torus represents the product of spheres $\Arm$ with its submanifold $\Pol$ and thin lines indices the fibers of the conformal barycenter closure which maps $x$ to $y$.}
\label{fig:FiberedTorus}
\end{figure}
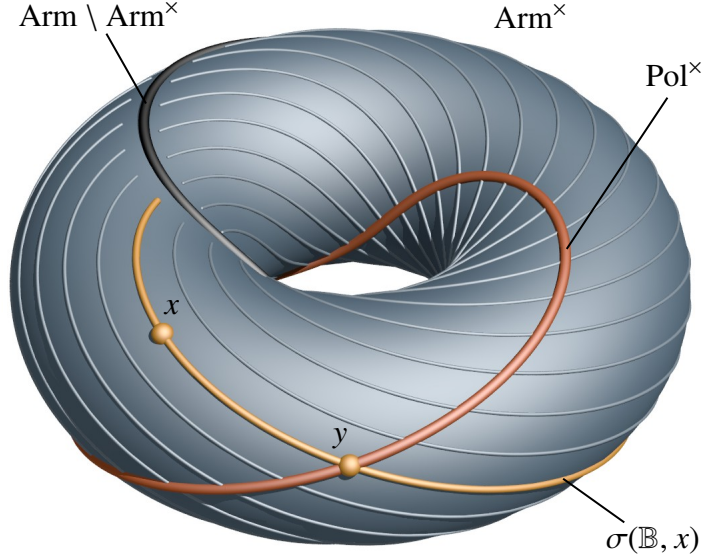

We have studied the shift map in detail in~\cite{CantarellaSchumacher2022}. In particular, for $\nabs{w} < 1$, the shift map is a conformal diffeomorphism $\Disk \rightarrow \Disk$ and $\Sphere \rightarrow \Sphere$. We now want to show that $\cl$ and $\op$ are inverse maps. This is not true everywhere, but it is true on subsets of full measure, which will be good enough for our purposes.
\begin{definition}
We let 
$\Arm^\times
	\ceq \set{x \in \Arm | \text{for all $i \neq j$: $x_i \neq x_j$ }}$
and
$	\Pol^\times
	\ceq \Pol \cap \Arm^\times$.
\label{def:pol times and arm times}
\end{definition}

We note that these are open and dense subsets of $\Arm$ and $\Pol$, and hence subsets of full measure in these spaces. Further, if $x \in \Arm^\times$ then $x$ is stable with respect to $r$. 


\begin{proposition}
The restrictions of the maps $\cl$ and $\op$ to $\Arm^\times$ and $\Disk \times \Pol^\times$ are maps $\cl \colon \Arm^\times \rightarrow \Disk \times \Pol^\times$ and $\op \colon \Disk \times \Pol^\times \rightarrow \Arm^\times$ and these maps are inverses of each other.  
\label{prop:closing and opening are inverses}
\end{proposition}
\begin{proof}
We first observe that for any fixed $w \in \Disk$, the M\"obius transformation $\Shift(w,-)$ is a diffeomorphism from $\Sphere^{d-1}$ to itself. Therefore, $y_i \neq y_j$ implies $\Shift(-w,y_i) \neq \Shift(-w,y_j)$ and $\op$ maps $\Disk \times \Pol^\times$ into $\Arm^\times$. Similarly, $x_i \neq x_j$ implies $\Shift(w_*,x_i) \neq \Shift(w_*,x_j)$ and so $\cl$ maps $\Arm^{\times}$ into $\Disk \times \Pol^\times$. 

We now prove that these maps are inverses of each other. Suppose that $y \in \Pol^\times$. This means that the conformal barycenter $w_*(y) = 0$. By equivariance of the conformal barycenter under hyperbolic isometries, we then have $w_*\pars[\big]{\Shift(-w,y)} = \Shift \pars[\big]{-w,w_*(y)} = \Shift(-w,0) = w$. 
Using that $\sigma(-w,-)$ is the inverse of $\sigma(w,-)$, we can now compute
\begin{align*}
	\cl \pars[\big]{ \op(w,y) } 
	= 
	\cl \pars[\big]{\Shift(-w,y)} 
	&= 
	\pars[\bigg]{ w_* \pars[\big]{\Shift(-w,y)},\Shift \pars[\Big]{w_*\pars[\big]{\Shift(-w,y)},\Shift(-w,y)} } 
	\\ 
	&= \pars[\Big]{ w,\Shift\pars[\big]{w,\Shift(-w,y)} } = (w,y).
\end{align*}
Conversely, suppose that $x \in \Arm^\times$. 
Then 
\[
	\op \pars[\big]{\cl(x)} 
	= 
	\op \pars[\Big]{w_*(x),\Shift \pars[\big]{{-w}_*(x),x} } 
	= 
	\Shift \pars[\Big]{ {-w}_*(x),\Shift\pars[\big]{w_*(x),x } } = x
	,
\]
which completes the proof.
\end{proof}
We will later show that $\Arm^\times$ and $\Disk \times \Pol^\times$ are smooth manifolds (\autoref{lem:poltimes is a mfld}) and that $\cl$ and $\op$ are diffeomorphisms between them (\autoref{theo:FastJ}). In fact, our construction will show that $\Arm^\times$ smoothly deformation retracts onto $\Pol^\times$, as shown in~\autoref{fig:FiberedTorus}. This will require more work. We start by using the change of variables formula to obtain a general formula for integration over $\Pol$. 

\begin{proposition}
\label{prop:sampling weights}
Suppose that $f \colon \Pol \rightarrow \R$ and $\chi \colon \Disk \rightarrow \R$ are integrable functions with $\int \chi(w) \, \dvol_{\Disk} = 1$, where~$\vol_{\Disk}$ is the Lebesgue measure on~$\Disk$. Then
\[
\int_{\Pol} f(y) \, \dvol_{\varrho,r}(y)  = \int_{\Pol^\times} f(y) \, \dvol_{\varrho,r}(y)  =
\int_{\Arm^\times} f(y_*(x)) \, K(\cl(x)) \, \dvol_{\varrho}(x),
\]
where $K(w,y) = \frac{\chi(w)}{J\op(w,y)}$, and $J\op$ is the Jacobian determinant with respect to the metric $g_{\varrho}$ on $\Arm$ and the metric $g_{\Disk} \times g_{\varrho,r}$ on $\Disk \times \Pol$, where $g_{\Disk}$ is the Euclidean metric on $\Disk$.
\end{proposition}

\begin{proof}
Suppose we have an integrable function $h \colon \Disk \times \Pol \to \R$ that we want to integrate with respect to the product measure~$\vol_{\Disk} \times \vol_{\varrho,r}$. We first observe that it suffices to integrate~$h$ over $\Disk \times \Pol^\times$ as the result will be the same. 

Assuming that $\cl \colon \Arm^\times \rightarrow \Disk \times \Pol^\times$ is a diffeomorphism (\autoref{theo:FastJ}), we can use the change of variables formula to pull back the integral of $h$ to $\Arm^\times$, writing
\[
\int_{\Disk \times \Pol^\times} h(w,y) \, \dvol_{\Disk}(w) \, \dvol_{\varrho,r}(y)
	=
	\int_{\Arm^\times} h(\cl(x)) \, J\cl(x) \, \dvol_{\varrho}(x).	
\label{eq:TransformationFormula}
\]
Here $J\cl(x)$ is the nonzero Jacobian determinant given by\[
  J\cl(x)
	=
	\sqrt{\det\pars[\big]{ 
		D\cl(x)^{*} D\cl(x)
	}},
\]
where we need to keep in mind that if $\cl(x) = (w,y)$, the differential $D\cl(x)$ is an invertible linear map $D\cl(x) \colon (T_{x} \Arm^\times,g_\varrho) \to (T_{w} \Disk \times T_y \Pol^\times,g_\Disk \times g_{\varrho,r})$. The adjoint $D\cl(x)^{*}$ is defined relative to these inner products. We cannot compute $J\cl(x)$ directly, but $\cl$ is a diffeomorphism. So, applying the change of variables formula to the inverse map $\op$, we know that 
\begin{align*}
	J\cl(x)
	=
	\frac{1}{J\op(\cl(x))},
	\quad
	\text{where}
	\quad
	J\op(w,y)
	=
	\sqrt{\det\pars*{ D\op(w,y)^{*} \, D\op(w,y)}} \neq 0
	.
\end{align*}
As above, $D\op(w,y)^*$ is the Riemannian adjoint. This yields
\begin{align*}
	\int_{\Disk \times \Pol^\times} h(w,y) \, \dvol_{\Disk}(w) \, \dvol_{\varrho,r}(y)
	&=
	\int_{\Arm^\times} \frac{ h(\cl(x))}{J\op(\cl(x))}\, \dvol_\varrho(x).
\end{align*}

We actually want to integrate the integrable function $f \colon \Pol \to \R$. So at this point we specialize this formula by assuming that we have chosen some integrable $\chi \colon \Disk \to \R$ and set $h(w,y) = \chi(w) f(y)$. We then get
\begin{equation}
\begin{aligned}
	\int_{\Pol^\times} f(y) \, \dvol_{\varrho,r}(y)
	&=
	\frac{\int_{\Disk \times \Pol^\times} \chi(w) f(y) \, \dvol_{\Disk}(w) \, \dvol_{\varrho,r}(y)}{\int_{\Disk} \chi(w) \, \dvol_{\Disk}(w)}
	\\
	&=
	\frac{1}{\int_{\Disk} \chi(w) \, \dvol_{\Disk}(w)}
	\int_{\Arm^\times} \frac{h(\cl(x))}{J\op(\cl(x))} \, \dvol_{\varrho}(x).
\end{aligned}
\label{eq:lateformula}
\end{equation}
Now $\cl(x) = (w_*(x),y_*(x))$, so
\begin{align*}
	\frac{h(\cl(x))}{J\op(\cl(x))} 
	=
	f(y_*(x)) \, K(w_*(x),y_*(x))
	\quad
	\text{where}
	\quad
	K(w,y) \ceq \frac{\chi(w)}{J\op(w,y)}.
\end{align*}
If we now add our assumption that $\int_{\Disk} \chi(w) \, \dvol_{\Disk} = 1$,~\eqref{eq:lateformula} becomes the statement of the Proposition, completing the proof.
\end{proof}

We have now expressed our integral over the entire space of closed polygons $y \in \Pol$ in terms of an integral over the subspace of open polygons $x \in \Arm^\times$. It is natural to ask whether we can extend the right hand integral to all of $\Arm$ since the complement of $\Arm^\times$ is a set of measure zero. We cannot: neither the map $\cl$ nor the weight function $K \circ \cl$ are well-defined on all of $\Arm \setminus \Arm^\times$ since this set includes some $x$ that do not have a conformal barycenter.

Of course, we want to compute integrals with respect to the normalized volume (or probability measure) on $\Pol$ given by $\PolProb = \frac{1}{\vol_{\varrho,r}(\Pol)} \dvol_{\varrho,r}$. By the law of large numbers, we may use reweighted sampling to estimate expectations over $\Pol$ as usual:

\begin{corollary}
 If $f \colon \Pol \rightarrow \R$ is integrable and if $x^{(j)}$ is a sequence of independent samples drawn from $\ArmProb$ on $\Arm$, then 
\begin{align}
\frac{
		\sum_{j=1}^N f(y_*(x^{(j)})) \, K(\cl(x^{(j)}))
	}{
		\sum_{j=1}^N K(\cl(x^{(j)})
	}
	\rightarrow
	\int_\Pol f(y) \,\dd \PolProb(y) 
	\quad
	\text{almost surely as $N \rightarrow \infty$.}
\label{eq:MonteCarloFormula}
\end{align}
\label{cor:MonteCarloSampling}
\end{corollary}


\section{Calculating $J\op(w,y)$}

Our eventual goal is to compute the weight function~$K(w,y)$. We start by proving:
\begin{theorem}\label{theo:FastJ}
For $(w,y) \in \Disk \times \Pol^{\times}$ we have
\begin{align}
	J\op(w,y)
	=
	\pars*{\frac{2}{1-\nabs{w}^{2}}}^{d}
	\,
	\frac{
		\det \pars*{ \sum_{i=1}^{n} r_{i} \, \pars[\big]{I_{d} - y_{i}\,y_{i}\transp} }
	}{
		\sqrt{
			\det \pars*{ \sum_{i=1}^{n} r_{i}^{2}/\varrho_{i}^{2} \, \pars[\big]{I_{d} - y_{i}\,y_{i}\transp} }
		}
	}
	\,
	\pars*{
		\prod_{i=1}^n  \frac{1-\nabs{w}^2}{\nabs{w + y_i}^2}
	}^{d-1}
	\neq 0.
\label{eq:Jop_formula}
\end{align}
For fixed $d$, this can be computed in $O(n)$ time and memory. Since this determinant does not vanish, $\op$ is a diffeomorphism from $\Disk \times \Pol^\times$ to $\Arm^\times$ and hence its inverse map $\cl$ is a diffeomorphism from $\Arm^\times$ to $\Disk \times \Pol^\times$.
\end{theorem}
This theorem is surprising, because the matrix $D\op(w,y)^* \, D\op(w,y)$ is of size $n \, (d-1) \times n \, (d-1)$, so we would expect its determinant to require $O(n^3)$ time and $O(n^2)$ memory to compute. However,~\eqref{eq:Jop_formula} only contains $d \times d$ determinants and it so may be evaluated in $O(n)$ time and memory. 

The proof of this theorem will require us to do some detailed matrix computations. Accordingly, we take a moment to establish a system of coordinates and maps, along with some notation.
\begin{lemma} 
We may extend $\op \colon \Disk \times \Pol^\times \subset \Disk \times \Sphere^n \to \Arm^\times \subset \Sphere^n$ to a smooth map $\hop \colon U \rightarrow (\R^d)^n$ where $U \subset (\R^d)^{n+1}$ is an open neighborhood of $\Disk \times \bar{\Disk}^n$ by extending $\Shift$ to $\hShift \colon V \to \R^d$, where $V \subset \R^d \times \R^d$ is an open neighborhood of $\Disk \times \bar{\Disk}$. 
\end{lemma}

\begin{proof} Using~\eqref{eq:shift} we may define the extension
\[
\hShift(w,z) 
	\ceq
	\frac{(1- \nabs{w}^2) \, z - (1 + \nabs{z}^2 - 2 \, \ninnerprod{w,z}) \, w}{1 - 2 \ninnerprod{w,z} + \nabs{w}^2 \, \nabs{z}^2}.
\]
The right hand side is in $\R^d$ and a smooth function defined for all $w,\,z \in \R^d$ where the denominator does not vanish. Now the denominator is
\[
	1 - 2 \, \ninnerprod{w, z} + \abs{w}^2 \, \abs{z}^2 \geq 1 - 2 \,\abs{w} \,\abs{z} + \abs{w}^2 \, \abs{z}^2 = (1 - \abs{w} \, \abs{z})^2,
\]
so it does not vanish when $\abs{w} \, \abs{z} < 1$. Thus we may define
\[
	V \ceq \bigcup_{0 < \lambda < 1} \pars*{ \lambda \, \Disk } \times \pars*{\tfrac{1}{\lambda} \, \Disk} \supset \Disk \times \bar{\Disk}.
\]
As a union of open sets, the set $V$ is clearly open. Now $\op(w,y) = \pars[\big]{\Shift(-w,y_1), \dotsc, \Shift(-w,y_n) }$, so we may let $\hop(w,y) \ceq \pars[\Big]{\hShift(-w,y_1), \dotsc, \hShift(-w,y_n)}$ and observe that $\hop$ is defined on the analogous set
\[
	U 
	\ceq 
	\bigcup_{0 < \lambda < 1} \pars*{ \lambda \, \Disk } \times \pars*{\tfrac{1}{\lambda} \, \Disk }^n \supset \Disk \times \bar{\Disk}^n. \qedhere
\]
\end{proof}
We will now assume $w \in \Disk$ and $y \in \Pol^\times$, define $x \in \Arm^\times$ by $x \ceq \op(w,y)$, and 
compute $D\op$ as the restriction of $D\hop$ to the linear subspace $T_{(w,y)} (\Disk \times \Pol^\times) \subset T_{(w,y)} (\R^d)^{n+1}$ in the usual coordinates on $(\R^d)^{n+1}$.  
Now we have the following commutative diagram:
\[
\begin{tikzcd}
	{T_{(w,y)}(\mathbb{B} \times \mathrm{Pol}^\times)} && {T_x \mathrm{Arm}^\times} \\
	\\
	{T_{(w,y)} (\mathbb{R}^d)^{n+1}} && {T_x(\mathbb{R}^d)^n}
	\arrow["{D\widetilde{\mathrm{op}}(w,y)}", from=3-1, to=3-3]
	\arrow["{D\mathrm{op}(w,y)}", from=1-1, to=1-3]
	\arrow[hook, from=1-1, to=3-1]
	\arrow[hook, from=1-3, to=3-3]
\end{tikzcd}
\]
Our next goal is to factor these maps through $T_y \Arm^\times$ and $T_y (\R^d)^n$. To so do, we examine the structure of $D\hop$ as a block matrix. We define $d \times d$ matrices
\begin{align}
	\tilde{a}_i \ceq -D_1 \hShift(-w,y_i) 
	\qand 
	\tilde{b}_i \ceq  D_2 \hShift(-w,y_i)
\label{eq:hatai and hatbi}
\end{align}
where $D_1$ and $D_2$ are the derivatives of $\hShift(-,-)$ with respect to the first and second (vector) arguments. Then we have
\begin{align}
D\hop(w,y) &= \begin{pmatrix}
\tilde{a}_1 & \tilde{b}_1 & 0 & \dotsm & 0 \\
\vdots & 0 & \ddots  & \ddots & \vdots \\
\vdots & \vdots & \ddots & \ddots & 0 \\
\tilde{a}_n & 0 & \dotsm & 0 & \tilde{b}_n 
\end{pmatrix} = \begin{pmatrix}
		\vect(\tilde{a})
		&
		\diag(\tilde{b})
	\end{pmatrix}.
\end{align}
Here $\vect(\tilde{a})$ denotes the $(n \, d) \times d$ matrix that results from the $d \times d$ matrices $\tilde{a}_{1},\dotsc,\tilde{a}_{n}$ by stacking them on top of each other and $\diag(\tilde{b})$ denotes the $(n\,d) \times (n\,d)$ block diagonal matrix with the $d \times d$ matrices $\hat{b}_i$ on the diagonal. Since $\hShift(-w,-)$ is a M\"obius transformation, its derivative $D_2 \hShift(-w,-)$ is invertible wherever it is defined. Thus the $\tilde{b}_i$ are invertible matrices for $(w,y) \in U$. 

We let $\tilde{Z} \ceq \diag(\tilde{b})$. This is $D\hop(w,-)$ evaluated at the point $y$. Since $\tilde{Z}$ is block diagonal and the blocks are invertible, $\tilde{Z}$ is invertible. Since $\tilde{Z}$ is the derivative of $\hop(w,-)$ at $y$, it must map $T_y (\R^d)^n \to T_x(\R^d)^n$. Analogously, we let $Z \ceq D\op(w,-)$ at $y$. It follows that $Z: T_y \Arm^\times \to T_x \Arm^\times$. Further, $Z$ is the block product of invertible linear maps $b_i \ceq D_2 \Shift(-w,y_i)$, which map $T_{y_i} \Sphere^{d-1}$ to $T_{x_i} \Sphere^{d-1}$, so $Z$ is invertible. It is also the restriction of $\tilde{Z}$ to $T_y \Arm^\times \subset T_y (\R^d)^n$. 

Now we can factor $D\hop(w,y) = \tilde{Z} \, \tilde{A}$ and $D\op(w,y) = Z \, A$, where $\tilde{A} \ceq \tilde{Z}^{-1} D\hop(w,y)$ and $A \ceq Z^{-1} D\op(w,y)$. 
Note that $\tilde{A}$ takes the following simple form:
\begin{align}
	\tilde{A} 
	&= \begin{pmatrix}
		\tilde{c}_1 & I_d & 0 & \dotsm & 0 \\
		\vdots & 0 & \ddots  & \ddots & \vdots \\
		\vdots & \vdots & \ddots & \ddots & 0 \\
		\tilde{c}_n & 0 & \dotsm & 0 & I_d
	\end{pmatrix} 
	\quad
	\text{with}
	\quad
	\tilde{c}_i \ceq \tilde{b}_i^{-1} \, \tilde{a}_i
	.
	\label{eq:hatci}
\end{align}
We summarize this construction by the commutative diagram
\[\begin{tikzcd}
	{T_{(w,y)}(\mathbb{B} \times \mathrm{Pol}^\times)} &&& {T_x \mathrm{Arm}^\times} \\
	&& {T_y \mathrm{Arm}^\times} \\
	{T_{(w,y)} (\mathbb{R}^d)^{n+1}} &&& {T_x(\mathbb{R}^d)^n} \\
	&& {T_y (\mathbb{R}^d)^n}
	\arrow["{D\widetilde{\mathrm{op}}(w,y)}", shorten >=38pt, no head, from=3-1, to=3-4]
	\arrow["{D\mathrm{op}(w,y)}", from=1-1, to=1-4]
	\arrow[hook, from=1-1, to=3-1]
	\arrow[hook, from=1-4, to=3-4]
	\arrow["A", from=1-1, to=2-3]
	\arrow["{\tilde{A}}", from=3-1, to=4-3]
	\arrow["Z", from=2-3, to=1-4]
	\arrow["{\tilde{Z}}", from=4-3, to=3-4]
	\arrow[hook, from=2-3, to=4-3]
	\arrow[shorten <=77pt, from=3-1, to=3-4]
\end{tikzcd}\]
This factorization of $D\op$ leads us to the following observation:

\begin{lemma}\label{lem:J=det(AA)det(ZZ)}
\begin{align*}
	J\op(w,y)^2
	=
	\det\pars[\Big]{ D\op(w,y)\adj \, D\op(w,y) }
	=
	\det( A\adj \, Z\adj \, Z\,A )
	=
	\det( A\adj \, A ) \, \det(Z\adj \, Z)
	.
\end{align*}
\end{lemma}

\begin{proof}
The first two equalities are definitions. The last follows from the fact that the spaces $T_{(w,y)} \Disk \times \Pol^\times$, $T_y \Arm^\times$ and $T_x \Arm^\times$ all have the same dimension. Therefore, we could write the linear maps $A$ and $Z$ as square matrices, allowing us to reorder before taking determinants. 
\end{proof}
Now we must compute $\det(Z\adj \, Z)$ and $\det(A\adj \, A)$. 
\begin{lemma}\label{lem:detSS} We have
\begin{align*}
	\det(Z\adj \, Z)
	=
	\prod_{i=1}^n 
	\pars*{
		\frac{1-\nabs{w}^2}{\nabs{w + y_i}^2}
	}^{2\,(d-1)}
	,
\end{align*}
which cannot vanish because $\nabs{y_i} = 1$ and $\nabs{w} < 1$. 
\end{lemma}
\begin{proof}
The map 
\begin{align*}
	Z\adj \, Z \colon 
	T_y \Arm = T_{y_1} \Sphere \times \dotsm \times  T_{y_n} \Sphere 
	\to 
	T_y \Arm = T_{y_1} \Sphere \times \dotsm \times T_{y_n} \Sphere
\end{align*}
is the direct product of $b_i^*b_i \colon T_{y_i} \Sphere \to T_{y_i} \Sphere$, where we recall
$b_i \colon T_{y_i} \Sphere \to T_{x_i} \Sphere$ is the map induced by $\tilde{b}_i$ and where the adjoint is with respect to the Riemannian metric $\varrho_i^2 \, g_\Sphere$ on $\Sphere$. Thus
\begin{align*}
	\det(Z\adj \, Z)
	=
	\prod_{i=1}^n \det(b_i^* b_i).
\end{align*}
Note that $\tilde{b}_i \colon \pars[\big]{ \R^d, \ninnerprod{\cdot,\cdot}_{\R^d}} \to \pars[\big]{ \R^d, \ninnerprod{\cdot,\cdot}_{\R^d}}$ is a conformal map as it is the derivative of the M\"obius transformation $\Shift(-w,-)$. Its conformal factor $\lambda_i$ with respect to the standard metric on $\R^d$ can be computed easily and is
\begin{align*}
	\lambda_i \ceq \frac{1-\nabs{w}^2}{1+ 2 \, \ninnerprod{w,y_i} + \nabs{w}^2} = \frac{1 - \nabs{w}^2}{\nabs{w + y_i}^2}
\end{align*}
using the fact that $\nabs{y_i} = 1$. Since $\tilde{b}_i$ maps $T_{y_i} \Sphere$ into $T_{x_i} \Sphere$, the induced mapping $b_i$ has the same conformal factor with respect to the standard metric $g_\Sphere$. Scaling $g_\Sphere$ by $\varrho_i^2$ does not change the conformal factor of $b_i$. The Riemannian adjoint $b_i^*$ is also conformal with the same conformal factor $\lambda_i$. Hence $b_i^* b_i$ has conformal factor $\lambda_i^2$. Since $T_{y_i} \Sphere$ is $(d-1)$-dimensional, we get $\det(b_i^* b_i) = \lambda_i^{2\,(d-1)}$, and the result follows. 
\end{proof}

We now set out to compute $\det(A^*A)$. We will do this in several steps. 

\begin{lemma}\label{lem:detAA}
Consider $\R^d \times (\R^d)^n$ equipped with the product metric of the standard inner product on $\R^d$ and the rescaled inner product 
$$
\ninnerprod{u,v}_\varrho = \sum_{i=1}^n \varrho_i^2 \, \ninnerprod{u_i,v_i}
$$
on $(\R^d)^n$ and let $P$ be the orthogonal projector onto the subspace $T_w \Disk \times T_y \Pol^\times$. Then  
\begin{align*}
	\det( A\adj \, A )
	=
	\det( P \, \tilde{A}\adj \, \tilde{A}\, P + I_{(n+1)\,d} - P) 
	.
\end{align*}
where adjoints are with respect to the inner product.
\end{lemma}
\begin{proof}
First, $A^*A$ and $P \, \tilde{A}\adj \, \tilde{A}\, P = P^* \tilde{A}^* \tilde{A} P$ are self-adjoint matrices and may be diagonalized. 
Let $e_1, \dots, e_{(n-1) \, d} \in T_w \Disk \times T_y \Pol^\times$ be the orthonormal eigenvectors and
$\lambda_1, \dots, \lambda_{d\, (n-1)}$ be the corresponding eigenvalues of  $A^*A$.
Because $A$ coincides with $\tilde{A} \, P$ on $T_w \Disk \times T_y \Pol^\times$, these are also eigenvectors and eigenvalues of $P \, \tilde{A}\adj \, \tilde{A}\, P$.
We complete this basis to an orthonormal basis for $\R^d \times (\R^d)^n$ by appending some vectors $e_{n\, (d-1) + 1}, \dotsc, e_{(n+1)\, d}$. These new vectors form an orthonormal basis for the null space of $P$, which is the image of the orthogonal projector $I_{(n+1)\,d} - P$, so we may write 
\[
P \tilde{A}^* \tilde{A} P + (I_{(n+1)\,d} - P) = \sum_{i=1}^{n \, (d-1)} \lambda_i \, e_i \, e_i^T + \sum_{i=n\, (d-1)+1}^{(n+1)\,d} \!\! e_i \, e_i^T,
\]
proving that $\det(P \tilde{A}^* \tilde{A} P + I_{(n+1)\,d} - P) = \prod_{i=1}^{n\,(d-1)} \lambda_i = \det(A^* A)$ as required. 
\end{proof}

\subsection{The projector $P$}

Our next task is to find a suitable representation for $P$. Note that $\Pol$ consists of the points in $(\R^d)^n$ where the polygon closes (so $\sum_{i=1}^n r_i y_i = 0$) and the $y_i$ have unit norm. This means that we can write $\Disk \times \Pol$ as the zero set of the map
\begin{align*}
	\varTheta \colon \Disk \times (\R^d)^n \to \R^d \times \R^n,
	\quad
	\varTheta(w,y) = \begin{pmatrix}
		\sum_{i=1}^n r_i \, y_i\\ 
		\frac{1}{2}( \nabs{y_1}^2 - 1)\\
		\vdots\\
		\frac{1}{2}( \nabs{y_n}^2 - 1)
	\end{pmatrix}
	.
\end{align*}
The derivative $B \ceq D\varTheta(w,y)$ is a linear mapping
\begin{align*}
	B \colon T_w \Disk \times T_y (\R^d)^n \cong \R^d \oplus (\R^{d})^{n} \to \R^d \times \R^n.
\end{align*}
Since $y \mapsto \sum_{i=1}^n r_i \, y_i$ is linear in the $y_i$, it is its own derivative; the derivative of $\frac{1}{2}(\nabs{y_i}^2 - 1)$ is easily computed to be $y_i^T$ (when $\nabs{y_i} = 1$). Therefore, with respect to this decomposition of vector spaces, $B$ can be written as a block matrix
\begin{align*}
	B
	=
	\begin{pmatrix}
	0      & r_1 \, I_d  & \cdots & \cdots & r_n \, I_d \\
	0      & y_1^T  & 0      & \cdots & 0 \\
	\vdots & 0      & \ddots & \ddots & \vdots \\
	\vdots & \vdots & \ddots & \ddots & 0 \\
	0      & 0      & \cdots & 0      & y_n^T    
	\end{pmatrix}
	=
	\begin{pmatrix} 0 & \varOmega \\ 0 & Y^T \end{pmatrix}
	,
\end{align*}
where $\varOmega \ceq
	\begin{pmatrix}
		r_1 \, I_d  &\cdots &r_n \, I_d
	\end{pmatrix}$
and $Y \ceq \diag(y_1, \dots, y_n)$ are matrices of size $d \times (d \,n)$ and $(n \, d) \times n$, respectively.
We claim that $B$ is surjective whenever $y \in \Pol^{\times}$. We prove this by establishing the following Lemma. 

\begin{lemma}
\label{lem:BBstar_invertible}
Let $(w,y) \in \Disk \times \Pol^{\times}$ and $B \ceq D\varTheta(w,y)$. Then the matrix $B \, B\adj$ is invertible, where $B\adj$ is the adjoint of $B$ 
with respect to the metric 
$g_1 \ceq \ninnerprod{\cdot,\cdot}_{\R^d} \times \ninnerprod{\cdot,\cdot}_{\varrho}$ 
on 
$\R^d \oplus \R^{n\,d}$
and 
with respect to the standard metric 
$g_2 \ceq \ninnerprod{\cdot,\cdot}_{\R^d} \times \ninnerprod{\cdot,\cdot}_{\R^n}$ 
on 
$\R^d \oplus \R^{n}$. 
\end{lemma}

\begin{proof} 
Recall that we defined the metric $\ninnerprod{\cdot,\cdot}_\varrho$ in~\autoref{sec:intro}. Now the definition of adjoint is that for all $u = (u_1,u_2) \in \R^d \oplus (\R^d)^n$ and all $v = (v_1,v_2) \in \R^d \oplus \R^{n}$ we have
$
	g_2(B \,u ,v)
	=
	g_1( u, B^*\, v)
$.
We compute
\begin{align*}
	g_2(B \,u ,v)
	&=
	\ninnerprod{ \varOmega \, u_2, v_1}_{\R^d} + \ninnerprod{ Y\transp \, u_2, v_2}_{\R^n}
	=
	\ninnerprod{ u_2, \varOmega\transp \, v_1 + Y \, v_2 }_{(\R^d)^n}
	\\
	&=
	\ninnerprod{ u_2, R^{-2} \, \varOmega\transp \, v_1 + R^{-2} \, Y \, v_2 }_{\varrho}
	=
	g_1 \pars[\Big]{ (u_1,u_2), ( 0, R^{-2} \, \varOmega\transp \, v_1 + R^{-2} \, Y \, v_2 ) },
\end{align*}
where $R$ is the blockdiagonal matrix of size $(n \, d) \times (n \, d)$ with blocks $\varrho_i \, I_d$ on the main diagonal.
That means the adjoint $B\adj$ has the following block structure:
\begin{align*}
	B\adj
	=
	\begin{pmatrix}
		0 
		& 0
		\\
		R^{-2}  \varOmega\transp
		& R^{-2} \, Y
	\end{pmatrix}
	,
\end{align*}
Note that $R^{-2} \, Y = Y \, \diag(\varrho_1^{-2}, \dotsc, \varrho_n^{-2})$ and $Y\transp \, Y = \diag(\nabs{y_1}^2, \dotsc, \nabs{y_n}^2) = I_n$, because $y \in \Pol^{\times}$. Hence we may write $B \, B\adj$ as the following block matrix:
\begin{align*}
	B \, B\adj
	&=
	\begin{pmatrix}
	0 & \varOmega \\
	0 & Y\transp 
	\end{pmatrix}
	\begin{pmatrix}
	0 & 0 \\
	R^{-2} \varOmega\transp & R^{-2} Y
	\end{pmatrix}
	=
	\begin{pmatrix}
		\varOmega \, R^{-2} \varOmega\transp
		& \varOmega \, R^{-2} Y
		\\ 
		Y\transp \, R^{-2} \varOmega\transp
		& \diag(\varrho^{-2})
	\end{pmatrix}
	.
\end{align*}
Any $2 \times 2$ block matrix whose lower right block is invertible may be written in UDL form
\begin{align}
	\begin{pmatrix}
	A & B \\
	C & D 
	\end{pmatrix}
	= 
	\begin{pmatrix}
	I & BD^{-1} \\
	0 & I 
	\end{pmatrix}
	\begin{pmatrix}
	A - BD^{-1}C & 0 \\
	0 & D 
	\end{pmatrix}
	\begin{pmatrix}
	I & 0 \\
	D^{-1} C & I 
	\end{pmatrix}
	,
	\label{eq:schur formula}
	\end{align}
where the upper left block of the center matrix is the Schur complement of the lower right block of the original matrix. 
For our matrix $B \, B\adj$, this Schur complement is 
\begin{align}
	\gamma 
	&\ceq 
	\varOmega \, R^{-2} \varOmega\transp 
	- 
	\varOmega \, R^{-2} Y \,\big(\diag(\varrho^{-2})\big)^{-1} Y\transp \, R^{-2} \varOmega\transp
	\label{eq:definition of gamma}
	\\
	&=
	\varOmega \, R^{-2} \varOmega\transp 
	- 
	\varOmega \, R^{-1} Y \, Y\transp \,  R^{-1} \varOmega\transp \notag
	\\
	&=
	\varOmega \, R^{-1} \pars[\big]{ I_{n \, d} - Y \, Y\transp } \,  R^{-1} \varOmega\transp \notag
\end{align}
and we may factorize $B \, B\adj$ as follows:
\begin{align}
	B \, B\adj
	=
	\begin{pmatrix}
		I_d 
		& \varOmega \, Y
		\\ 
		0
		& I_n
	\end{pmatrix}
	\begin{pmatrix}
		\gamma
		& 0 
		\\ 
		0 
		& \diag(\varrho^{-2})
	\end{pmatrix}
	\begin{pmatrix}
		I_d 
		& 0 
		\\ 
		Y\transp \, \varOmega\transp 
		& 
		I_n
	\end{pmatrix}
	.
	\label{eq:BBTFactorization}
\end{align}
The two outer factors are triangular matrices with ones on the main diagonals. Thus they are invertible. The center matrix is block diagonal and the block $\diag(\varrho^{-2})$ is invertible by assumption. Hence $B \, B\adj$ is invertible if and only if the Schur complement $\gamma$ is invertible.

Keeping in mind that $R^{-1}$ and $I_{n \, d} - Y\, Y\transp$ are block-diagonal, and that multiplying by $\varOmega$ takes a weighted sum over rows or columns, this matrix can be written as a weighted sum of the diagonal blocks of $I_{n \, d} - Y \, Y\transp$:
\begin{align*}
	\gamma 
	= 
	\varOmega \, R^{-1} \pars[\big]{ I_{n \, d} - Y \, Y\transp } \,  R^{-1} \varOmega\transp 
	= 
	{\textstyle\sum_{i=1}^n} r_i^2 / \varrho_i^2  \, \pars[\big]{ I_d - y_i \, y_i\transp }.
\end{align*}
We see that $\gamma$ is symmetric and positive semidefinite. Assume $\gamma$ is not positive definite.
Then there is a unit vector $V \in \Sphere$ with 
\begin{align*}
	0 = \ninnerprod{V , \gamma \,V}
	=
	{\textstyle \sum_{i=1}^{n}} r_{i}^{2} / \varrho_{i}^{2} \, \pars{1 - \ninnerprod{y_{i},V}^{2} }
	.
\end{align*}
This can only happen if $y_{i} \in \set{-V,V}$ for all $i \in \set{1,\dotsc,n}$. Since we require $n \geq 3$, this implies that there must be at least one pair of indices $i \neq j$ with $y_{i}=y_{j}$. But this contradicts the condition $y \in \Pol^{\times}$. We note that this is why we introduced $\Pol^\times$ in the first place. 

Hence $\gamma$ must be positive definite, showing that $B\,B\adj$ is invertible and that $B\adj$ is surjective.
\end{proof}

As a side effect, by the implicit function theorem (or transversality of $\varTheta$ to $0$) we have shown the following:
\begin{lemma} 
The set $\Disk \times \Pol^{\times}$ is a smooth manifold with tangent space $T_{(w,y)}(\Disk \times \Pol^{\times}) = \ker(B)$.
\label{lem:poltimes is a mfld}
\end{lemma}
We are now in a position to accomplish the main goal of this section:
\begin{proposition}
The orthoprojector $P$ onto $T_{(w,y)}(\Disk \times \Pol)$
with respect to the scaled metric $\ninnerprod{\cdot,\cdot}_{\R^d} \times \ninnerprod{\cdot,\cdot}_{\varrho}$ 
is given by
\begin{align}
	P
	&=
	\begin{pmatrix}
		I_d
		& 0
		\\
		0
		&E
	\end{pmatrix}
	,
	\quad \text{where} \quad
	E \ceq Q - R^{-2} Q \, \varOmega\transp \,\gamma^{-1} \varOmega \, Q, 
	\quad
	\text{and}
	\quad
	 Q \ceq I_{n \, d} - Y \, Y\transp.
	\label{eq:P}
\end{align}
Further, $E$ is the orthogonal projector with respect to $\ninnerprod{\cdot,\cdot}_{\varrho}$ from $(\R^d)^n$ onto $T_y \Pol^\times$.
\end{proposition}

\begin{proof} Since the tangent space is the kernel of $B$ and $B$ is surjective, it follows that
\begin{align*}
	P = I_{(n+1) \, d} - B\adj \, (B \, B\adj)^{-1} \, B.
\end{align*}
Inverting $B \, B\adj$ is easy with the factorization from \eqref{eq:BBTFactorization}:
\begin{align*}
	(B \, B\adj)^{-1}
	=
	\begin{pmatrix}
		I_d 
		& 0 
		\\ 
		-Y\transp \, \varOmega\transp 
		& I_n
	\end{pmatrix}
	\begin{pmatrix}
		\gamma^{-1} 
		& 0 
		\\ 
		0 
		& \diag(\varrho^2)
	\end{pmatrix}
	\begin{pmatrix}
		I_d 
		& -\varOmega \, Y 
		\\ 
		0 
		& 
		I_n
	\end{pmatrix}	
	.
\end{align*}
It is only a matter of some algebra to obtain the expression for $P$ above.
Since $P$ is an orthogonal projector, it has to satisfy $P\adj = P$ and $P \, P = P$. For $E$ this implies
\begin{align}
	E\adj = E \qand E \, E = E,
	\label{eq:E_ortho}
\end{align}
where $E\adj$ denotes the adjoint with respect to the metric $\ninnerprod{\cdot,\cdot}_\varrho$. Thus, $E$ is also an orthogonal projector.
\end{proof}

\subsection{Determinant of $P \, \tilde{A}^* \,  \tilde{A} \, P + I_{(n+1) \, d} - P$}
Recall that $\tilde{c}_{i} \ceq \tilde{b}_{i}^{-1} \, \tilde{a}_{i}$ where $\tilde{a}_i$ and $\tilde{b}_i$ were defined in~\eqref{eq:hatai and hatbi}, and $\tilde{c}_{i}$ was defined in~\eqref{eq:hatci}. To compute the determinant of $P \, \tilde{A}^* \,  \tilde{A} \, P + I_{(n+1) \, d} - P$, we will eventually need more information about $\tilde{c}_i$. So we start by deriving an explicit formula. 
\begin{proposition}
When $y \in \Arm$, each $\tilde{c}_i$ is in the form 
\begin{align}
\tilde{c}_i
	=
	\frac{2}{1-\nabs{w}^{2}}
	\,
	\pars*{
		(1 + \ninnerprod{w,y_i}) \, I_{d} - (y_i + w) \,y_i\transp 
	}
	.
\label{eq:hatci formula}
\end{align}
and further, 
\begin{align}
	\pars[\big]{I_d - y_i \, y_i\transp}\, \tilde{c}_i = \tilde{c}_i.
\label{eq:qc equals c}
\end{align}
\end{proposition}
\begin{proof}
We start by recalling that 
\begin{align*}
	\hShift(w,z) 
	\ceq
	\frac{(1- \nabs{w}^2) \, z - (1 + \nabs{z}^2 - 2 \, \ninnerprod{w,z}) \, w}{1 - 2 \ninnerprod{w,z} + \nabs{w}^2 \, \nabs{z}^2}.
\end{align*}
while $\tilde{a}_i = -D_1 \hShift(-w,y_i)$ and $\tilde{b}_i = D_2 \hShift(-w,y_i)$ (see \eqref{eq:hatai and hatbi}). Mechanically differentiating the formula above using the identity $(u/v)' = (1/v)\,(u' - (u/v) \, v')$, one eventually gets
\begin{align*}
	D_1 \hShift(s,z) 
	&=
	\frac{1}{1- 2 \, \ninnerprod{s,z} + \nabs{s}^2 \, \nabs{z}^2}
	\pars*{
		2 \, s \, z\transp - 2\,z \, s\transp - (1 + \nabs{z}^2 - 2 \, \ninnerprod{s,z}) \, I_{d}
		+
		2 \, \hShift(s,z) \, (z\transp - \nabs{z}^2 \, s\transp)
	}
	\\
	D_2
	\hShift(s,z) 
	&=
	\frac{1}{1- 2 \, \ninnerprod{s,z} + \nabs{s}^2 \, \nabs{z}^2}
	\,
	\pars[\Big]{
		(1- \nabs{s}^2) \, I_{d} - 2 \, s \, (z\transp - s\transp)
		+
		2 \, \hShift(s,z)
		\,
		(s\transp - \nabs{s}^2 \, z\transp)
	}
	.
\end{align*}
For the rest of the proof, we will suppress the index on $y_i$ for clarity, writing $y$ instead. Substituting $z = y$ with $\nabs{y} = 1$, this simplifies to
\begin{align*}
	D_1
	\hShift(s,y) 
	&=
	\frac{2}{1- 2 \, \ninnerprod{s,y} + \nabs{s}^2}
	\pars*{
		s \, y\transp - y \, s\transp - (1 - \ninnerprod{s,y}) \, I_{d}
		+
		\hShift(s,y) \, (y\transp - s\transp)
	}
	\\
	D_2
	\hShift(s,y) 
	&=
	\frac{2}{1- 2 \, \ninnerprod{s,y} + \nabs{s}^2}
	\,
	\pars*{
		\frac{1- \nabs{s}^2}{2} \, I_{d} - s \, (y\transp - s\transp)
		+
		\hShift(s,y)
		\,
		(s\transp - \nabs{s}^2 \, y\transp)
	}
	.
\end{align*}
With the abbreviation $x \ceq \hShift(-w,y)$, we obtain
\begin{align*}
	a \ceq -D_1 \hShift(-w,y) 
	&=
	\frac{2}{1 + 2 \, \ninnerprod{w,y} + \nabs{w}^2}
	\pars*{
		w \, y\transp - y \, w\transp + (1 + \ninnerprod{w,y}) \, I_{d}
		-
		x \, (y\transp + w\transp)
	}
	\\
	b \ceq
	D_2 
	\hShift(-w,y) 
	&=
	\frac{2}{1 + 2 \, \ninnerprod{w,y} + \nabs{w}^2}
	\,
	\pars*{
		\frac{1- \nabs{w}^2}{2} \, I_{d} + w (y\transp + w\transp)
		-
		x
		\,
		(w\transp + \nabs{w}^2 \, y\transp)
	}
	.
\end{align*}
We claim that when $\nabs{y} = 1$, we have $b^{-1} a = c$, where 
\begin{align*}
	c
	\ceq
	\frac{2}{1-\nabs{w}^{2}}
	\,
	\pars*{
		(1 + \ninnerprod{w,y}) \, I_{d} - (y+w)\,y\transp
	}
	.
\end{align*}
Observe that $x = \hShift(-w,y)$ is in the plane of $w$, $y$, and $0$. Therefore, from the formulas above, we see that $a$, $b$, and $c$ map this plane to itself. In fact, $a$, $b$, and $c$ also preserve the orthogonal complement of this plane. If we let $e_1 = y$ and $e_2$ be a unit vector in this plane perpendicular to~$y$, we may complete this to an orthonormal basis for $\R^d$. In this basis, the matrices $a$, $b$ and $c$ are block-diagonal with upper left $2 \times 2$ blocks and lower right $(d-2) \times (d-2)$ blocks. 

For convenience, we define $\alpha = 2/(1 + 2 \, \ninnerprod{w,y} + \nabs{w}^2)$. Then we compute
\begin{align*}
	a = \begin{pmatrix} 
		a_{11} & a_{12} & 0 \\
		a_{21} & a_{22} & 0 \\
		0 & 0 & \alpha (1 + \ninnerprod{w,y}) I_{d-2}
	\end{pmatrix}
	,\;
		b = \begin{pmatrix}
		b_{11} & b_{12} & 0 \\
		b_{21} & b_{22} & 0 \\
		0 & 0 & \alpha \frac{1 - \nabs{w}^2}{2} I_{d-2}
	\end{pmatrix}
	,\;
	c = \begin{pmatrix}
		c_{11} & c_{12} & 0 \\
		c_{21} & c_{22} & 0 \\
		0 & 0 & \frac{2 (1 + \ninnerprod{w,y})}{1 - \nabs{w}^2} I_{d-2}
	\end{pmatrix}.
\end{align*}
It suffices to show that $b\, c = a$, which we may do block-by-block. It is already clear that the product of the lower right blocks of $b$ and $c$ is equal to the corresponding block of $a$. We are left checking the upper left ($2 \times 2$) blocks. If $w = r \cos(\theta) \, e_1 + r \sin(\theta) \, e_2$, the upper left blocks of $a$, $b$, and $c$ are:
\begin{align*}
	\begin{pmatrix} 
	a_{11} & a_{12} 
	\\
	a_{21} & a_{22} 
	\end{pmatrix} 
	&= 
	\begin{pmatrix}
		\frac{4 r^2 \sin^2(\theta ) \, (r \cos (\theta )+1)}{ \pars{ r^2+2 r \cos (\theta )+1 }^2 } 
		& 
		-\frac{4 \, r \,  \sin (\theta ) \, \pars{ r \cos (\theta )+1 }^2}{ \pars{r^2 + 2 \, r \cos (\theta ) + 1 }^2 } 
		\\
		-\frac{2 \, r \, \sin (\theta ) \, \pars{r^2 \cos (2\,\theta )+2\,r \cos (\theta )+1}}{ \pars{ r^2+2\,r \cos (\theta)+1 }^2} 
		& 
		\frac{2 \, \pars{r \, \cos (\theta )+1} \, \pars{ r^2 \cos (2\,\theta )+2\,r \cos (\theta)+1 }}{\pars{ r^2+2\,r \cos (\theta )+1 }^2}
	\end{pmatrix}
	, 
	\\
	\begin{pmatrix} 
	b_{11} & b_{12} 
	\\
	b_{21} & b_{22} 
	\end{pmatrix} 
	&= 
	\begin{pmatrix}
		-\frac{\pars{r^2-1} \, \pars{r^2 \cos (2\,\theta )+ 2 \, r \cos (\theta )+1}}{\pars{r^2 + 2 \, r \cos (\theta ) + 1}^2} 
		& 
		\frac{2 \, r \, \pars{r^2-1} \, \sin (\theta ) \, (r \cos (\theta )+1)}{\pars{r^2 + 2 \, r \cos (\theta )+1}^2} 
		\\
		-\frac{2 \,r \, \pars{r^2-1} \, \sin (\theta ) \, (r \cos (\theta )+1)}{\pars{r^2 + 2 \, r \cos (\theta )+1}^2} 
		&
		-\frac{\pars{r^2-1} \,\pars{r^2 \cos (2\,\theta ) + 2 \, r \cos (\theta )+1}}{\pars{r^2 + 2 \, r \cos (\theta ) + 1}^2} 
	\end{pmatrix}
	,
	\\
	\begin{pmatrix} 
		c_{11} & c_{12} 
		\\
		c_{21} & c_{22} 
	\end{pmatrix} 
	&= 
	\begin{pmatrix}
		0 & 0 \vphantom{\frac{4 r^2 \sin ^2(\theta ) (r \cos (\theta )+1)}{ \pars{ r^2+2 r \cos (\theta )+1 }^2 } }
		\\
		\frac{2 r \sin (\theta )}{r^2-1} & - \frac{2 (r \cos (\theta )+1)}{r^2-1}
	\end{pmatrix}
	.
\end{align*}
Now it is easy to check that the product of the upper left blocks of $b$ and $c$ is equal to the corresponding block of $a$,
as required for~\eqref{eq:hatci formula}. To prove~\eqref{eq:qc equals c}, we just observe that in our basis,
\[
	I_d - yy\transp 
	= 
	\begin{pmatrix}
		0 & 0 & 0 \\
		0 & 1 & 0 \\
		0 & 0 & I_{d-2}
	\end{pmatrix}
	.
	\qedhere
\]
\end{proof}

We are now ready to work on the main goal of this section:
\begin{lemma}\label{lem:detPAAP} For $(w,y) \in \Disk \times \Pol^\times$, we have
\begin{align*}
	\det( P \, \tilde{A}\adj \, \tilde{A}\, P + I_{(n+1)\,d} - P ) 
	=
	\pars*{\frac{2}{1-\nabs{w}^{2}}}^{2 \, d}
	\,
	\frac{
		\det \pars[\Big]{ 
			\sum_{i=1}^{n} r_{i} \, \pars[\big]{I_{d} - y_{i}\,y_{i}\transp} 
		}^2
	}{
		\det \pars[\Big]{
			\sum_{i=1}^n r_i^2 / \varrho_i^2 \, \pars[\big]{I_d - y_i \, y_i\transp }
		}
	}
	\neq 0.
\end{align*}
\end{lemma}

\begin{proof}
We start by recalling from~\eqref{eq:hatci} that
\begin{align*}
\tilde{A} \ceq
	\tilde{Z}^{-1} \, D\hop(w,y)
	=
	\begin{pmatrix}
		C
		&
		I_{n \, d}
	\end{pmatrix}
	,\quad \text{where} \quad
	C \ceq \vect(\tilde{c}).
\end{align*}
Thus, we obtain
\begin{align*}
	\tilde{A} \, P
	=
	\begin{pmatrix}
		C
		& 
		E
	\end{pmatrix}
	\qand
	(\tilde{A} \, P)\adj
	=
	P \, \tilde{A}\adj
	=
	\begin{pmatrix}
		C\adj
		\\
		E\adj
	\end{pmatrix}
	.
\end{align*}
This, combined with \eqref{eq:E_ortho} allows us to compute
\begin{align*}
	P \, \tilde{A}\adj \, \tilde{A} \, P + I_{(n+1) \, d} - P
	=
	\begin{pmatrix}
		C\adj C 
		&
		C\adj \,E
		\\
		E\adj \, C
		&
		I_{n \, d} - E + E\adj \,E
	\end{pmatrix}
	=
	\begin{pmatrix}
		C\adj C 
		&
		C\adj \,E
		\\
		E\adj \, C
		&
		I_{n \, d }
	\end{pmatrix}
	.
\end{align*}
The determinant can be computed by Schur's formula~\eqref{eq:schur formula}; utilizing \eqref{eq:E_ortho} once more, we obtain:
\begin{align*}
	\det
	\begin{pmatrix}
		C\adj C 
		&
		C\adj \,E
		\\
		E\adj \, C
		&
		I_{n \, d}
	\end{pmatrix}
	=
	\det \pars[\big]{ C\adj C - C\adj \, E\,  I_{n \, d}^{-1} \, E\adj \, C } \cdot \det \pars[\big]{ I_{n \, d} }
	=
	\det \pars[\big]{ C\adj C - C\adj \, E\, C }
	.
\end{align*}
Now $Q = \diag \pars[\big]{I_d - y_1 \, y_1\transp\, \dotsc, I_d - y_n \, y_n\transp}$ while $C = \vect(\tilde{c}_1, \dotsc, \tilde{c}_n)$, so~\eqref{eq:qc equals c} implies that $Q\, C = C$. Since $Q$ is symmetric, this implies $C\transp Q = C\transp$. Now $C$ is a map from $\R^d$ with the standard inner product to $(\R^d)^n$ with the inner product  $\ninnerprod{u,v}_\varrho = \sum_i \varrho_{i}^2 \, \ninnerprod{u_i,v_i}_{\R^d}$ (from~\autoref{sec:opening and closing polygons}). If $u \in \R^d$ and $v \in (\R^d)^n$, the adjoint $C\adj$ is defined by 
\begin{align*}
	\ninnerprod{C\adj v,u}_{\R^d} = \ninnerprod{v,Cu}_{\varrho} 
	= 
	\sum_{i=1}^n \varrho_i^2 \, \ninnerprod{v_i,\tilde{c}_i \,u_i}_{\R^d} 
	= 
	\sum_{i=1}^n \ninnerprod{(\tilde{c}_i\transp \, \varrho_i^2) \,v_i, u_i}_{\R^d} 
	= 
	\ninnerprod{C\transp \, R^2 \, v,u}_{\R^d}.
\end{align*}
It follows, substituting the definition of $E$, that
\begin{align*}
	C\adj \, C  - C\adj \, E\, C
	&=
	\cancel{C\adj \, C} - \cancel{C\adj \, Q\, C} + C\adj \, R^{-2}  Q \, \varOmega\transp \,\gamma^{-1} \varOmega \, Q \, C \\
	&= 
	C\transp \, \cancel{R^2 \, R^{-2}} Q \, \varOmega\transp \,\gamma^{-1} \,\varOmega \, Q \, C \\
	&=
	C\transp \, \cancel{Q} \, \varOmega\transp \,\gamma^{-1} \varOmega \, \cancel{Q} \, C
	=
	C\transp\, \varOmega\transp \,\gamma^{-1} \varOmega \, C.
\end{align*}
Since $C\transp\, \varOmega\transp$, $\gamma^{-1}$ (see~\eqref{eq:definition of gamma}), and $\varOmega \, C$ are $d \times d$ matrices, we have
\begin{align*}
	\det(C\adj \, C  - C\adj \, E\, C)
	=
	\det \pars[\big]{ C\transp\, \varOmega\transp \,\gamma^{-1} \,\varOmega \, C }
	=
	\det \pars[\big]{ C\transp\, \varOmega\transp} \,\det\pars[\big]{\gamma^{-1}} \, \det \pars[\big]{\varOmega \, C}
	=
	\frac{\det(\varOmega \, C)^2}{\det(\gamma)}.
\end{align*}
Finally, we use~\eqref{eq:hatci formula} and the fact that $y$ is a closed polygon (so $\sum_{i=1}^n r_i \, y_i = 0$) to compute
\begin{align*}
	\varOmega \, C 
	=
	\sum_{i=1}^{n} r_{i} \,\tilde{c}_{i} 
	= 
	\frac{2}{1-\nabs{w}^{2}} \, \pars*{ \sum_{i=1}^{n} r_{i} \, \pars[\big]{I_{d} - y_{i}\,y_{i}\transp}
	-
    \pars[\Big]{ \innerprod{w,\textstyle \cancel{\sum_{i=1}^{n} r_{i} \, y_{i}}}} \, I_{d} 
	- 
	w  \, \pars[\Big]{\textstyle \cancel{\sum_{i=1}^{n} r_{i} \, y_{i}}}\transp}.
\end{align*}
Now we can see that $\varOmega \, C$ (like $\gamma$) is a weighted sum of the projectors $I_d - y_i \, y_i\transp$. Since $y \in \Pol^\times$, this sum is full rank (see the proof of \autoref{lem:BBstar_invertible}). This shows that the determinant does not vanish.
\end{proof}

We can now prove~\autoref{theo:FastJ}.

\begin{proof}[Proof of~\autoref{theo:FastJ}]
The formula for $J\op$ in~\eqref{eq:Jop_formula} is an immediate consequence of \autoref{lem:J=det(AA)det(ZZ)}, \autoref{lem:detSS}, \autoref{lem:detAA}, and \autoref{lem:detPAAP}. We have mentioned already that both $\sum_{i=1}^{n} r_{i} \, \pars[\big]{I_{d} - y_{i}\,y_{i}\transp} $ and $\sum_{i=1}^{n} r_{i}^{2}/\varrho_{i}^{2} \, \pars[\big]{I_{d} - y_{i}\,y_{i}\transp}$ are positive definite matrices of size $d \times d$. 
They can be computed in $O(n \, d^{2})$ time and $O(n\,d +  d^{2})$ memory. Their determinants can be computed, e.g., by Cholesky factorization in $O(d^{3})$ time and $O(d^{2})$ memory. In total, we require $O((n+d) \, d^{2})$ time and $O((n +d)\,d)$ memory. 
\end{proof}

\subsection{Modified Monte-Carlo sampling}

One severe problem with this sampling seems to be that $J\op(w,y)$ may become quite small sometimes, resulting in some nasty outliers of $1/J\op(w,y)$.
Let's have a closer look at the formula from \autoref{theo:FastJ}. After reordering we have
\begin{align*}
	J\op(w,y)
	=
	2^{d}
	\,
	\pars*{ 1-\nabs{w}^2 }^{n\,(d-1) -d}
	\frac{
		\det \pars*{ \sum_{i=1}^{n} r_{i} \, \pars[\big]{I_{d} - y_{i}\,y_{i}\transp} }
	}{
		\sqrt{
			\det \pars*{ \sum_{i=1}^{n} r_{i}^{2}/\varrho_{i}^{2} \, \pars[\big]{I_{d} - y_{i}\,y_{i}\transp} }
		}
	}
	\,
	\pars*{
		\prod_{i=1}^n \, \nabs{w + y_i}^{2} 
	}^{1-d} 
	.
\end{align*}
The number $1-\nabs{w}^2$ is raised to the power $n\,(d-1) -d$. For $n \gg d$ it may become very small even if $w$ is not very close to the boundary of the disk. This motivates us to choose the weight function $\chi \colon \Disk \to \R$ in~\autoref{prop:sampling weights} as follows:
\begin{align*}
	\chi(w) \ceq \lambda \, 2^{d} \, \pars*{ 1-\nabs{w}^2 }^{n\, (d-1)-d}.
\end{align*}
where $\lambda$ is a constant chosen so that $\int_\Disk \chi(w) \, \dvol_{\Disk} = 1$. 
We obtain the following weights:
\begin{align}
	\label{eq:final sampling weights}
	K(w,y) \ceq \frac{\chi(w)}{J\op(w,y)}
	=
	\lambda \, 
	\frac{
		\sqrt{
		\det \pars*{ \sum_{i=1}^{n} r_{i}^{2}/\varrho_{i}^{2} \, \pars[\big]{I_{d} - y_{i}\,y_{i}\transp} }
		}
	}{
		\det \pars*{ \sum_{i=1}^{n} r_{i} \, \pars[\big]{I_{d} - y_{i}\,y_{i}\transp} }
	}
	\,
	\pars*{
		\prod_{i=1}^n \, \nabs{w + y_i}^2
	}^{d-1}
	,
\end{align}
noting that a direct computation yields 
\begin{align*}
	\lambda = \frac{1}{2^d \, \pi^{d/2}} 
	\, 
	\frac{\Gamma\pars[\big]{ (n-1) \, (d-1) + \frac{d}{2}}}{\Gamma\pars[\big]{(n-1) \, (d-1)}}
	.
\end{align*}
However in practice (e.g., when using these weights for sampling as in~\autoref{cor:MonteCarloSampling}) it is more convenient to ignore $\lambda$ entirely, as the factors of $\lambda$ in the numerator and denominator of~\eqref{eq:MonteCarloFormula} cancel.

\section{The quotient by $\SO(d)$ }

In many contexts, one is interested in shapes of polygons but not their poses in $\R^d$. Therefore, it is desirable to identify configurations which are related by a rigid motion and study the resulting moduli space of equivalence classes of polygons. This is the point of view taken by mathematicians who study polygons in $\R^3$ via their symplectic structure~\cite{MillsonKapovich1996, Hausmann1996, Mandini2014,Kamiyama1999b}, and by the present authors in our previous papers~\cite{CDSU,CSS}. We now see how to transfer our work to this moduli space.

The diagonal action of the special orthogonal group $\SO(d)$ on $(\R^d)^n$ is an action by isometries of our metric $\ninnerprod{u,v}_\varrho = \sum_{i=1}^n \varrho_i^2\, \ninnerprod{u_i,v_i}$. Since each $\Sphere$ is invariant under this action, it restricts to an action by isometries on $(\Arm, g_{\Arm})$. Since the closure condition $\sum_{i=1}^n r_i \, y_i = 0$ is also $\SO(d)$-invariant, the action further restricts to an action by isometries on $(\Disk \times \Pol, \ninnerprod{\cdot,\cdot} \times g_{\Pol})$. 
Since $\op$ is M\"obius-equivariant and $\SO(d)$ is a subgroup of the M\"obius group, the diffeomorphism $\op \colon \Disk \times \Pol^\times \to \Arm^\times$ is $\SO(d)$-equivariant. These actions are faithful, but they are not free if the $y_i$ lie in some lower-dimensional subspace of $\R^d$. So now we slightly restrict our attention to avoid these troublesome configurations.

\begin{definition}
We define $\Arm^\diamond$ to the be subset of $\Arm^\times$ where the points $x_1, \dots, x_n$ do not lie in any affine hyperplane in $\R^d$. We define $\Pol^\diamond$ to be the subset of $\Pol^\times$ where the $y_i$ do not lie in any linear hyperplane in $\R^d$. 
\end{definition}

\begin{lemma}  The closure map $\cl$ and the opening map $\op$ are diffeomorphisms between $\Arm^\diamond$ and $\Disk \times \Pol^\diamond$. If $n > d$, then $\Arm \setminus \Arm^\diamond$ and $\Pol \setminus \Pol^\diamond$ are nullsets.
\label{lem:pol and arm diamond}
\end{lemma}

\begin{proof} 
Suppose $(w,y) \in \Disk \times \Pol^\diamond$. We claim that $x = \op(w,y) \in \Arm^\diamond$. Suppose not. Then the $x_i$ lie in an affine hyperplane in $\R^d$. Since the $x_i$ also lie in the unit sphere $\Sphere$, they lie in some $S^{d-2}$ formed by the intersection of the affine hyperplane with $\Sphere$. It follows that the $x_i$ and their conformal barycenter $w$ lie on some unique $S^{d-1} \subset \R^d$, and that this sphere $S$ intersects the unit sphere at right angles. 

Now the $y_i$ (and the origin) are the image of the $x_i$ (and $w$) under a M\"obius transformation. Therefore, they lie in the image of $S$ under this M\"obius transformation. Since this image is either a sphere or a hyperplane, since it meets the unit sphere $\Sphere$ at right angles, and since  it contains the origin, it must be a hyperplane. This contradicts our assumption that $(w,y) \in \Disk \times \Pol^\diamond$. Thus $\op$ maps $\Disk \times \Pol^\diamond$ into $\Arm^\diamond$. 

The argument that $\cl$ maps $\Arm^\diamond$ into $\Disk \times \Pol^\diamond$ is quite similar. Suppose $x \in \Arm^\diamond$, but $\cl(x)$ is not in $\Pol^\diamond$. Then the $y_i$ all lie in some hyperplane $H$. Now the opening map $\op(w,y)$ is a M\"obius transformation (of the $y_i$), so it maps $H$ to some $S^{d-1}$ containing $w$. This sphere also contains the $x_i$. Since the $x_i$ are also on the unit sphere, they lie in the intersection of two different $d-1$ spheres, which is in turn contained in some affine hyperplane. But this contradicts our assumption that $x \in \Arm^\diamond$. 

Now $\cl$ and $\op$ are diffeomorphisms between the larger sets $\Arm^\times$ and $\Disk \times \Pol^\times$, so they remain diffeomorphisms on the open subsets $\Arm^\diamond$ and $\Disk \times \Pol$. Now observe that $\Arm \setminus \Arm^\diamond$ is the disjoint union of $\Arm \setminus \Arm^\times$ and $\Arm^\times \setminus \Arm^\diamond$. We know that $\Arm \setminus \Arm^\times$ is a nullset (\autoref{def:pol times and arm times}). We claim that $\Arm^\times \setminus \Arm^\diamond$ is a nullset, too: The probability that the points $x_1, \dots, x_d$ span an affine hyperplane in $\R^d$ is one, while the probability that $x_{d+1}$ lies in the same affine hyperplane is zero. 

Finally, since the diffeomorphism $\cl$ maps $\Arm^\times \setminus \Arm^\diamond$ to $\Pol^\times \setminus \Pol^\diamond$, the latter is also a nullset. It follows as above (again, see~\autoref{def:pol times and arm times}) that $\Pol \setminus \Pol^\diamond$ is a nullset as well.  
\end{proof}

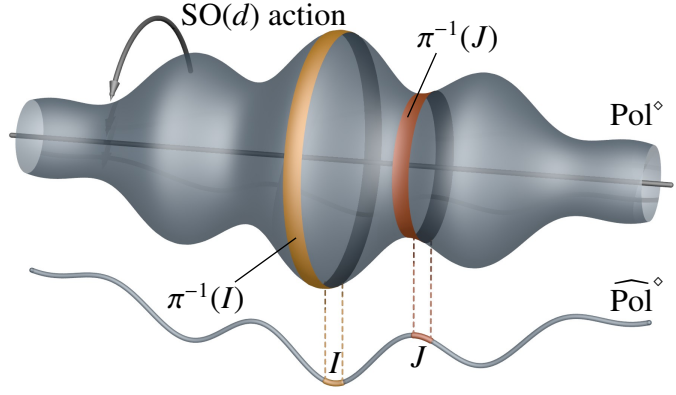
\begin{figure}[tbp]
\begin{center}
	\capstart
	\begin{tikzpicture}
	    \node[anchor=south west,inner sep=0] (image) at (0,0,0) {\includegraphics[width=0.6\textwidth,trim={0 0 0 0}]{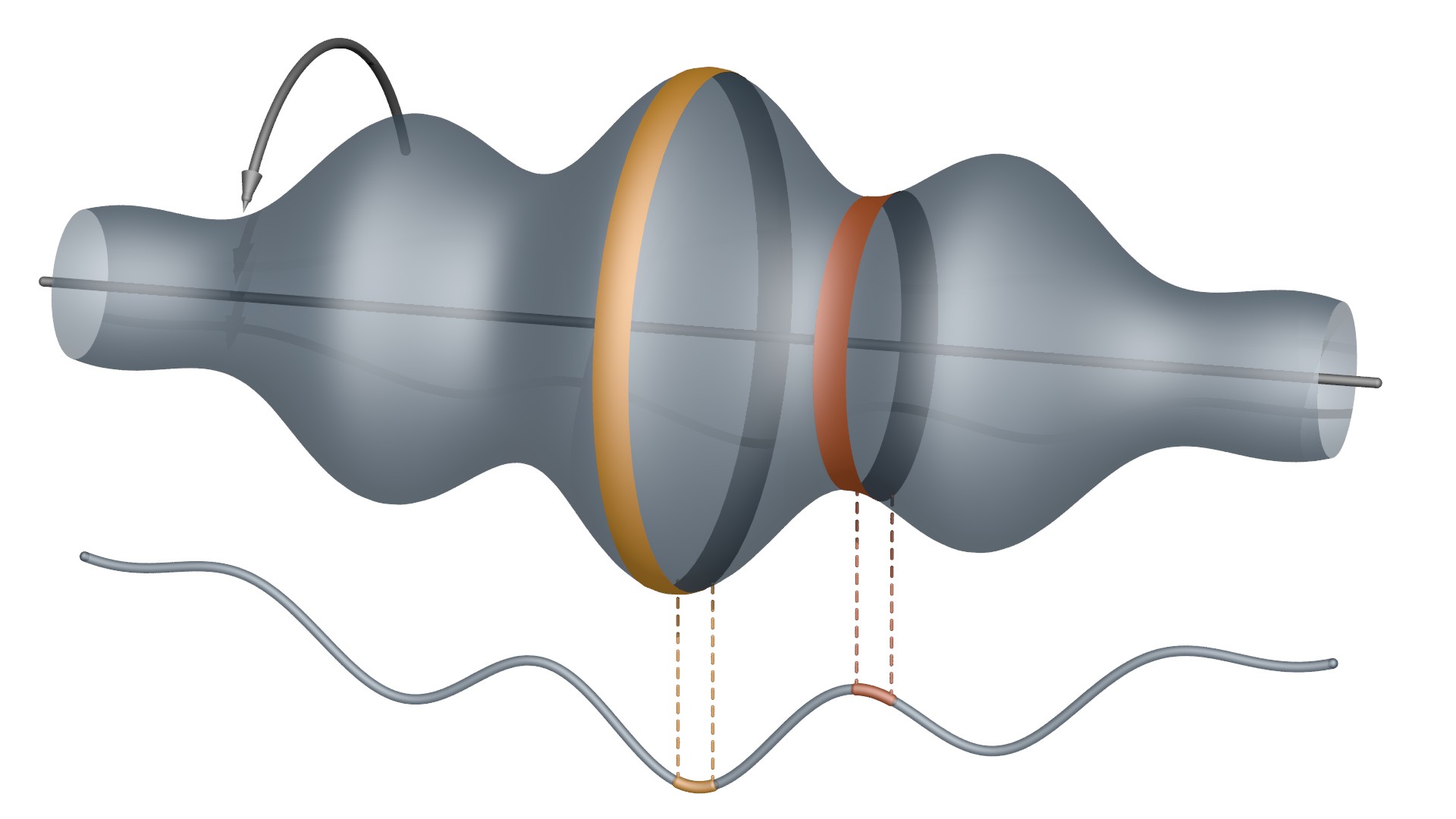}};
	    \begin{scope}[x={(image.south east)},y={(image.north west)}]
	        
			\node at ((0.38 ,0.95,0) {$\SO(d)$ action};
			\node at ((0.479,0.09 ,0) {$I$};
			\node at ((0.3  ,0.25 ,0) {$\QuotMap^{-1}(I)$};
			\draw[line width=0.5pt] (0.34,0.29,0) -- (0.425,0.4,0);
			\node at ((0.595,0.105,0) {$J$};
			\node at ((0.65 ,0.9  ,0) {$\QuotMap^{-1}(J)$};
			\draw[line width=0.5pt] (0.62,0.86,0) -- (0.58,0.7,0);
			\node at ((0.9  ,0.71 ,0) {$\Pol^\diamond$};
			\node at ((0.9  ,0.26 ,0) {$\PolQuot^\diamond$};
	    \end{scope}
	\end{tikzpicture}

\caption{Here the surface of revolution represents $\Pol^{\diamond}$ and the rotations about the axis represents the action of $\SO(d)$ on $\Pol^{\diamond}$ by isometries. The Riemannian quotient space $\PolQuot^\diamond$ is represented by the meridian curve. We see that the quotient metric (which measures products between vectors in the directions orthogonal to the rotations in the metric of $\Pol$) is the arclength metric on the meridian as a space curve shown below the surface. 
The yellow arc $I$ and orange arc $J$ on $\PolQuot^\diamond$ are subsets of equal volume according to the Riemannian probability measure $\PolQuotProb$. 
But they have very different volumes in the pushforward of $\PolProb$ under the quotient map as the corresponding yellow and orange annuli in $\Pol^{\diamond}$ have very different areas. An $\SO(d)$-invariant function on $\Pol^{\diamond}$ has a constant value on the fiber over any point in $\PolQuot^\diamond$. Therefore, it descends to a function on $\PolQuot^\diamond$ and may be integrated there with respect to either $\PolQuotProb$ or the pushforward $\QuotMap_\push \PolProb$. But we must expect the results to differ.}
\label{fig:Loxodrome}
\end{center}
\end{figure}

Now $\SO(d)$ acts smoothly, freely, and by isometries on $(\Pol^\diamond,g_{\varrho,r})$, so this is the total space of a principal bundle $\Pol^\diamond \overset{\QuotMap}{\longrightarrow} \Pol^\diamond/\SO(d)$. We will use the notation $\PolQuot^\diamond \ceq \Pol^\diamond/\SO(d)$.

The quotient space is an open manifold. One natural choice of measure on this space is the pushforward measure $\QuotMap_\push \PolProb$ of $\PolProb$ along the quotient map. 
This has the desirable feature that for any $\SO(d)$-invariant integrable $f \colon \Pol \to \R$ we may unambiguously define $\hat{f} \colon \PolQuot \to \R$ such that $f = \hat{f} \circ \pi$ and get 
\begin{align}
\int_{\Pol} f \, \dd \PolProb = \int_{\PolQuot} \hat{f} \, \dd (\QuotMap_\push \PolProb).
\end{align}
However, this is~\emph{not} the usual choice in the literature on polygon spaces. Instead, we let the quotient space $\PolQuot^\diamond$ inherit the Riemannian quotient metric from $(\Pol^\diamond,g_{\varrho,r})$ and construct the corresponding Riemannian volume measure $\PolQuotVol$. In turn, this volume measure defines a Riemannian probability measure
\[
	\PolQuotProb \ceq \frac{1}{\PolQuotVol \pars[\big]{\PolQuot^\diamond}} \PolQuotVol
\] on $\PolQuot^\diamond$ (where we extend the measure from $\PolQuot^\diamond$ to $\PolQuot$ by zero). The pushforward measure and the metric measure are really different from one another; \autoref{fig:Loxodrome} illustrates the issue.

\begin{proposition}
\label{prop:QuotientCorrectionFactor}
Suppose that $\varphi \colon \PolQuot \rightarrow \R$ is an integrable function. Then
\begin{align*}
	\int_{\PolQuot} \varphi(y) \, \dd \PolQuotVol(y) 
	= 
	\int_{\PolQuot^\diamond} \varphi(y) \, \dd \PolQuotVol(y) 
	= 
	\int_{\Arm^\diamond} \varphi \pars[\big]{ \QuotMap( y_*(x)) }\, \hat{K}(\cl(x)) \, \dd \vol_{\varrho}(x)
	.
\end{align*}
Here, if $\lambda_1(y), \dots, \lambda_d(y)$ are the eigenvalues of $\varSigma(y) \ceq \sum_{i=1}^n \varrho_i^2 \,y_i \, y_i\transp$, and $K(w,y)$ are the sampling weights from~\eqref{eq:final sampling weights}, we define
\begin{align}
\label{eq:complicated quotient sampling weights}
	\hat{K}(w,y) 
	\ceq 
	\frac{K(w,y)}{\vol(\SO(d))} 
	\,
	\pars*{
		 \prod_{1 \leq k < l < n} 
		 \frac{\lambda_{k}(y)+\lambda_{l}(y)}{2} 
	}^{-1/2} .
\end{align}
\end{proposition}
 
\begin{proof}
The coarea formula (see \cite[p.160]{zbMATH05034390} for a formulation in terms of differential forms or \cite[Section~3.4.2]{Evans_2018} for one in terms of Hausdorff and Lebesgue measures) tells us that
\begin{align*}
	\int_{\Pol^\diamond} \varphi(\QuotMap(y)) \, J\QuotMap(y) \, \dd \vol_{\varrho,r}(y)
	=
	\int_{\PolQuot^\diamond} 
	\pars*{ 
		\int_{\QuotMap^{-1}(z)} \varphi(\QuotMap(y)) \, \dvol_{\varrho,r}^{\QuotMap^{-1}(z)} (y) 
	} 
	\, \dd \PolQuotVol (z),
\end{align*}
where 
$
	J\QuotMap(y)
	=
	\sqrt{ \det \pars*{ \dd \QuotMap(y) \, \dd \QuotMap(y)\adj}}.
$
Since $\SO(d)$ acts freely, each fiber $\QuotMap^{-1}(z)$ is parametrized by $\SO(d)$, and the submanifold volume $\dvol_{\varrho,r}^{\QuotMap^{-1}(z)}$ on the fiber with respect to the ambient metric $g_{\varrho,r}$ is also the $\dim \SO(d) = d(d-1)/2$-dimensional Hausdorff measure $\mathcal{H}^{d(d-1)/2}$ with respect to this metric.

In the quotient metric, $\QuotMap$ is a Riemannian submersion, so $J\QuotMap(y) = 1$.  Observing that $\varphi \circ \QuotMap$ is constant on each fiber of $\QuotMap$, we can simplify the above to
\begin{align*}
	\int_{\Pol^\diamond} \varphi(\QuotMap(y)) \, \dd \vol_{\varrho,r}(y)
	=
	\int_{\PolQuot^\diamond} 
	\varphi(z) \, 
	\mathcal{H}^{d(d-1)/2} (\QuotMap^{-1}(z))
	\, \dd \PolQuotVol(z).
\end{align*}
The fiber $\QuotMap^{-1}(z)$ is an orbit of the $\SO(d)$ action. We now set out to calculate the volume of this orbit. There is no reason to expect that all $\SO(d)$ orbits will have the same volume, so we start by fixing some $y \in \QuotMap^{-1}(z)$ and compute the orbit volume in terms of $y$. 

We start by parametrizing the orbit of $y$ by the smooth map $f \colon \SO(d) \to \Pol$ given by $f(Q) = (Q \, y_1,\dotsc,Q\,y_n) = \diag(Q) \, y$. The orbit volume is then computed by the integral 
\begin{align*}
	\int_{\SO(d)} \sqrt{\det \pars[\big]{ Df(Q)\adj \, Df(Q)}} \, \dd \vol(Q).
\end{align*}
The map $f$ is $\SO(d)$ equivariant, i.e., we have $f(Q \, Q') = \diag(Q) \, f(Q')$ for all $Q$, $Q' \in \SO(d)$. Applying the derivative with respect to $Q'$ on both sides, we obtain with the chain rule that $Df(Q \, Q') \, Q  = \diag(Q) \, Df(Q')$. For $Q' = I_d$ this can be rewritten as $Df(Q)  = \diag(Q) \, Df(I_d) \, Q\adj$.
Thus we have 
\begin{align*}
	\det \pars[\big]{ Df(Q)\adj \, Df(Q)} 
	&= 
	\det \pars[\Big]{ (\diag(Q) \, Df(I_d) \, Q\adj)\adj \, \diag(Q) \, Df(I_d) \, Q\adj}
	\\
	&=
	\det \pars[\Big]{ Q\adj \, Q \, Df(I_d)\adj \, \diag(Q)\adj \, \diag(Q) \, Df(I_d)}
	=
	\det \pars[\big]{ Df(I_d)\adj \, Df(I_d) }.
\end{align*}
Here we exploited that $\diag(Q)$ is an isometry with respect to $g_{\varrho,r}$ and that $\SO(d)$ acts on itself isometrically with respect to the Frobenius metric.
Therefore, the integrand is constant and it suffices to evaluate it at $Q = I_d$. 

At $I_d$, the tangent space to $\SO(d)$ consists of the skew-symmetric matrices. Further, if $\xi$ and $\eta$ are $d \times d$ skew-symmetric matrices, then
\begin{align*}
	Df(I_d) \, \xi 
	= 
	\begin{pmatrix}
		\xi \, y_1
		\\ 
		\vdots 
		\\ 
		\xi \, y_n
	\end{pmatrix}
	\qand
	Df(I_d) \, \eta 
	= 
	\begin{pmatrix}
		\eta \, y_1
		\\ 
		\vdots 
		\\ 
		\eta \, y_n
	\end{pmatrix}
	.
\end{align*}
We now observe that (using the cyclic invariance of trace and the skew-symmetry of $\xi$),
\begin{gather*}
	\ninnerprod{\xi, Df(I_d)\adj \, Df(I_d) \, \eta}_{\text{Frob}}
	=
	\ninnerprod{Df(I_d) \, \xi, Df(I_d) \, \eta}_\varrho
	=
	\sum_{i=1}^n \varrho_i^2 \, \ninnerprod{\xi \, y_i, \eta \, y_i}
	\\	
	=
	\sum_{i=1}^n \varrho_i^2 \, \tr \pars[\Big]{ y_i\transp \, \xi\transp \,\eta \, y_i }
	=
	\sum_{i=1}^n \varrho_i^2 \, \tr \pars[\Big]{ \eta \, \pars[\big]{y_i \, y_i\transp} \, \xi\transp}
	=
	-\tr \pars[\big]{ \eta \, \varSigma \, \xi }
	.
\end{gather*}
We now choose $\eta$ and $\xi$ in order to calculate entries of the $(d(d-1)/2) \times (d(d-1)/2)$ matrix $Df(I_d)\adj \, Df(I_d)$. We first choose an orthonormal basis for $\R^d$ which diagonalizes the symmetric $d \times d$ matrix $\varSigma$ and assume the diagonal entries are $\lambda_1, \dotsc, \lambda_d$. With respect to this basis for $\R^d$, the $d \times d$ skew-symmetric matrices have an orthonormal basis given by matrices 
\begin{align*}
\xi(i,j)_{k,l} = \begin{cases}
\frac{1}{\sqrt{2}} & \text{if $(k,l)=(i,j)$,} \\
-\frac{1}{\sqrt{2}} & \text{if $(k,l)=(j,i)$,} \\
0 & \text{otherwise.}
\end{cases}
\end{align*}
Note that multiplying by $\xi(i,j)$ on the right swaps the $i$-th and $j$-th columns, multiplying the $j$-th by $-1/\sqrt{2}$ and the $i$-th by $+1/\sqrt{2}$. We may then compute 
\begin{align*}
	\pars[\big]{Df(I_d)\adj \, Df(I_d)\adj}_{(i,j),(k,l)} 
	\ceq 
	-\tr \pars[\Big]{ \xi(i,j) \, \varSigma \, \xi(k,k) } 
	= -\tr \pars[\Big]{ \xi(k,l) \, \xi(i,j) \,\varSigma} 
	= \frac{ \lambda_{i} + \lambda_{j} }{2} \, \updelta_{(i,j),(k,l)} 
\end{align*}
as the matrix product $\xi(k,l) \, \xi(i,j)$ has diagonal elements if and only if $(i,j) = (k,l)$, in which case it has $-\frac{1}{2}$ in positions $k$ and $l$. Thus $Df(I_d)\adj \, Df(I_d)$ is a diagonal matrix with diagonal entries $\frac{ \lambda_{k} + \lambda_{l} }{2}$. 
It follows that the orbit volume is 
\[
	\vol(\SO(d) \, y) 
	= 
	\vol(\SO(d)) \, \det \pars[\big]{ Df(I_d)\adj \, Df(I_d) }
	=
	\vol(\SO(d)) \, \pars*{ 
		\prod_{1 \leq k < l \leq n} 
		\frac{\lambda_{k}+\lambda_{l}}{2} 
	}^{1/2}.
	\qedhere
\]
\end{proof}

As before, we can immediately write down the formula for Monte Carlo sampling:

\begin{corollary}
If $f \colon \PolQuot \rightarrow \R$ is integrable and if $x^{(j)}$ is a sequence of independent samples drawn from $\ArmProb$ on $\Arm$, then 
\begin{align}
\frac{
		\sum_{j=1}^N f \pars[\big]{ \QuotMap(y_*(x^{(j)})) } \, \hat{K}(\cl(x^{(j)}))
	}{
		\sum_{j=1}^N \hat{K}(\cl(x^{(j)})
	}
	\rightarrow
	\int_{\PolQuot} f \, \dd \PolQuotProb
	\quad
	\text{almost surely as $N \rightarrow \infty$.}
\label{eq:MonteCarloFormula2}
\end{align}
\label{cor:MonteCarloQuotientSampling}
\end{corollary}
We note that as in~\eqref{eq:final sampling weights}, we may ignore constant factors in the definition of the sampling weights $\hat{K}$ when using~\eqref{eq:MonteCarloFormula2}, as they will cancel in the numerator and denominator. Therefore, we can ignore $\vol(\SO(d))$ and the factors of $2$ in the denominator, effectively using
\begin{align}
	\hat{K}(w,y) 
	= 
	\frac{
		\sqrt{
		\det \pars*{ \sum_{i=1}^{n} r_{i}^{2}/\varrho_{i}^{2} \, \pars[\big]{I_{d} - y_{i}\,y_{i}\transp} }
		}
	}{
		\det \pars*{ \sum_{i=1}^{n} r_{i} \, \pars[\big]{I_{d} - y_{i}\,y_{i}\transp} }
	}
	\,
	\pars*{
		\prod_{i=1}^n \, \nabs{w + y_i}^2
	}^{d-1}
	\,
	\pars*{
		\prod_{1 \leq k < l \leq n}
		\pars[\big]{\lambda_{k}(y)+\lambda_{l}(y)}
	}^{-1/2}
\label{eq:final quotient sampling weights}
\end{align}
as sampling weights for $\PolQuotProb$ on $\PolQuot$.

We note that if $n<d$, the polygon always lies in a lower-dimensional subspace of $\R^d$ and the group $\SO(d)$ no longer acts freely. One can apply essentially the same techniques as above, and obtain a similar formula containing only the $\frac{1}{2} \, ( \lambda_{i} + \lambda_{j})$ for which $\lambda_{i} + \lambda_{j} > 0$. As this case is not of much relevance to Monte-Carlo sampling, we leave the details to the interested reader.


\section{Experiments}

We now give the results of some sample computations which show our method at work. All of these computations were made using our open-source implementation of the sampling algorithm~\cite{Cantarella_CoBarS_-_Conformal,Cantarella_CoBarSLink_-_A}. 

We have mentioned above that the symplectic volume by viewing the quotient space $\PolQuot$ of polygons in $\R^3$ as the symplectic reduction of $\Arm$ by the diagonal action of $\SO(3)$ at the zero fiber (as in~\cite{MillsonKapovich1996}) corresponds in our model to setting $\varrho_i = \sqrt{r_i}$. In~\autoref{fig:nonequilateral}, we test this statement explicitly by computing the distribution of the chord skipping the first three edges of an hexagon with unit sides and a hexagon with sidelengths $(1,1/2,3/2,1,1,1)$ in the $\PolQuotProb$ measure with $\varrho = \sqrt{r}$. 

The distribution of the length of the chord skipping $k$ edges in an $n$-gon under $\PolSympRedProb$ in $\R^3$ may be computed by conditioning randomly selected open polygons with $k$- and $n-k$ edges on having the same failure to close vector~(cf.~\cite{Moore2005,Varela2009}). These distributions are complicated piecewise-polynomial functions, but they have been known for some time~\cite[Section 5]{Dutka}. Using this method, we computed the distributions $f_{\text{eq}}$ and $f_{\text{neq}}$ of the length of the chord skipping three edges in an equilateral hexagon and a hexagon with edgelengths $r = (1,1/2,3/2,1,1,1)$ to be 
\begin{align}
f_{\text{eq}}(\ell) = \begin{cases} 
			\ell^2, & \text{if $0 \leq r \leq 1$,} \\
			\frac{(\ell-3)^2}{4}, & \text{if $1 \leq r \leq 3$,} \\
			0, & \text{otherwise.}
		  \end{cases}
\quad\qand\quad
f_{\text{neq}}(\ell) = \begin{cases}
 \frac{4 \ell^2}{5} & \text{if $0<\ell \leq 1$,} \\
 \frac{2(3-\ell)}{5}  & \text{if $1\leq \ell \leq 2$,} \\
 \frac{2 (3-\ell)^2}{5} & \text{if $2\leq \ell \leq 3$,} \\
 0, & \text{otherwise.}
\end{cases}
\label{eq:hexagon chords}
\end{align}

\begin{figure}[thp]
\begin{center}%
\capstart
\begin{overpic}[width=0.5\textwidth]{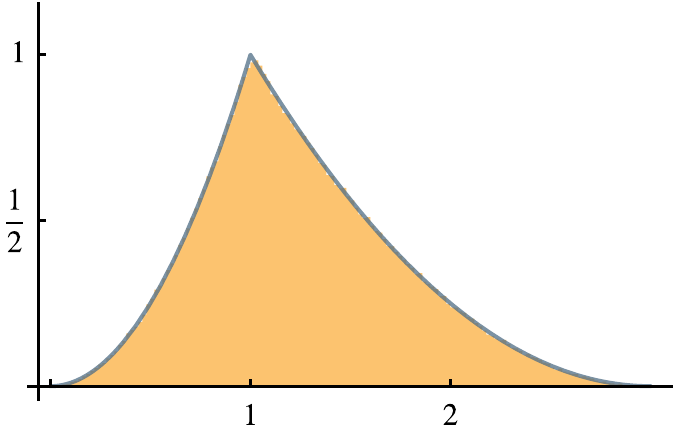}%
\put(50,50){$f_{\text{eq}}$}%
\end{overpic}%
\begin{overpic}[width=0.5\textwidth]{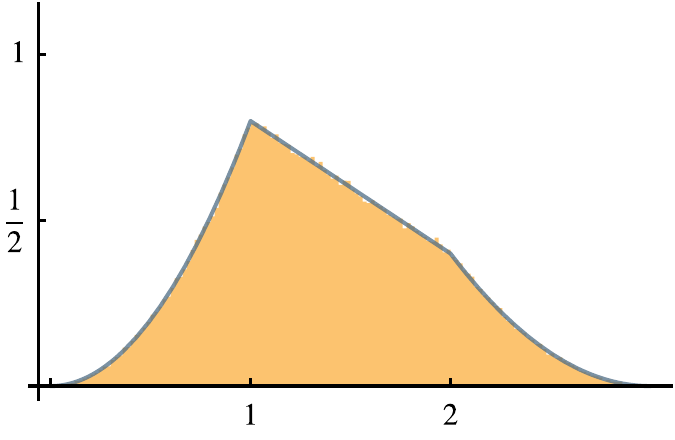}%
\put(50,50){$f_{\text{neq}}$}%
\end{overpic}%
\end{center}

\caption{On the left hand side, we see the theoretical pdf $f_{\text{eq}}$ for the length of a chord skipping three edges in an equilateral hexagon from the left of~\eqref{eq:hexagon chords}, plotted with a histogram of 1 million samples from $\PolSympRedProb$ weighted by the sampling weights in~\eqref{eq:final quotient sampling weights}. On the right hand side, we see the theoretical pdf $f_{\text{neq}}$ for the length of chord skipping the first three edges in a hexagon with sidelengths $r = (1,1/2,3/2,1,1,1)$ from the right of~\eqref{eq:hexagon chords}, also plotted with a histogram of 1 million samples from $\PolSympRedProb$ weighted by the sampling weights in~\eqref{eq:final quotient sampling weights}. We see that we resolve the peak clearly in both cases. We note that we also tested the results against a variant of the moment polytope sampling method of~\cite{CDSU}, confirming the result both times.}
\label{fig:nonequilateral}
\end{figure}

Next, we illustrate the difference between $\PolProb$ and $\PolQuotProb$ for equilateral tetragons in $\R^3$ by considering the distribution of the length of the chord joining vertices separated by two edges. The results are shown in~\autoref{fig:4gon chords}.

\begin{figure}[h]
\begin{center}
\capstart
\includegraphics[width=0.5\textwidth]{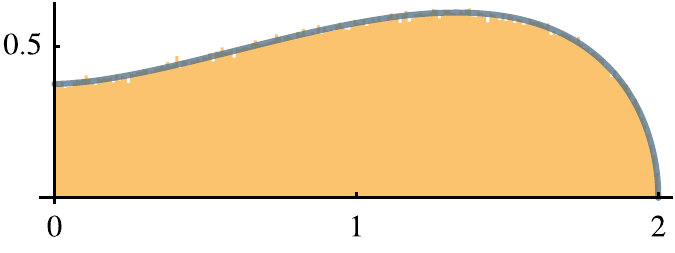}%
\includegraphics[width=0.5\textwidth]{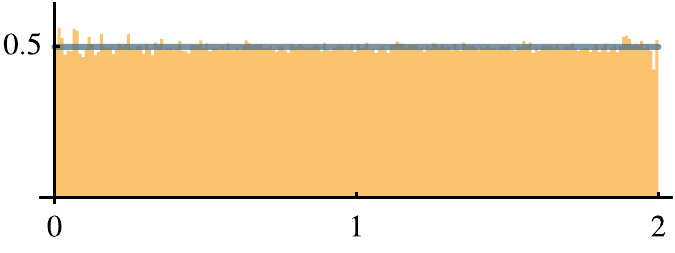}%
\end{center}
\caption{A detailed calculation reveals the probability distribution function of the length $\ell$ of the chord joining vertices separated by two edges of an equilateral tetragon in $\R^3$ in $\PolProb$ to be proportional to $8 \, \ell \sqrt{4-\ell^2} \, E\left(-(\ell^2-4)^2/(16 \,\ell^2)\right)$ where $E$ is the complete elliptic integral of the second kind \cite[\href{https://dlmf.nist.gov/19.2.E8}{(19.2.8)}]{NIST:DLMF}. A histogram of 1 million weighted samples computed using the sampling weights in~\eqref{eq:final sampling weights} is plotted along with this function at left. The probability distribution of the same length in $\PolQuotProb$ is known to be $1/2$. At right, we see this function plotted along with a histogram of 1 million weighted samples computed using the sampling weights in~\eqref{eq:final quotient sampling weights}. We see that in both cases agreement between theory and computation is very good, and note as well that the results are quite different.}
\label{fig:4gon chords}
\end{figure}

We last present a performance comparison for our algorithm versus the Action-Angle Method (AAM)~\cite{CDSU} and the Progressive Action-Angle Method (PAAM)~\cite{CSS} for equilateral $n$-gons in~$\R^3$ with respect to $\PolQuotProb$. We are comparing reweighted sampling with direct sampling, so it would be unfair to measure performance by the number of sampled polygons per second. Instead, we use each sampler to estimate the expected (squared) radius of gyration
\begin{align*}
	R^2(v) = \frac{1}{n^2} \sum_{i,j = 1}^n \, \nabs{v_i - v_j}^2 = \frac{1}{n} \, \sum_{i}^n \nabs{v_i - \bar v}^2,
\end{align*}
where $v_1,\dotsc,v_n$ are the vertex positions of the polygon and $\bar v$ is their mean.
We use standard techniques to estimate confidence intervals and stop when the radius of the $99\%$ confidence interval is less than $0.1\%$ of the sample mean. In \autoref{fig:perf} we report the number of samples required as well as timings.  For this integrand, CoBarS requires (asympotically) about $25\%$ more samples than AAM/PAAM. However, CoBarS has time complexity $O(n)$, while
AAM requires $O(n^{2.5})$ time and PAAM requires $O(n^{2})$ time. Therefore, the small number of additional samples is quickly amortized as $n$ increases, with crossover around $n = 50$.

\begin{figure}[t]
\begin{center}
\capstart
\includegraphics{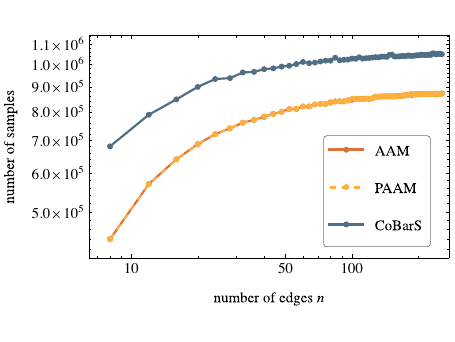}
\hfill
\includegraphics{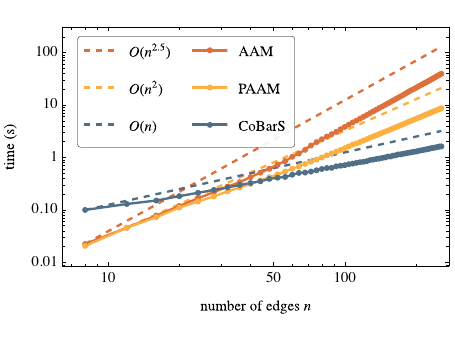}
\end{center}
\caption{Number of samples (left) and time (right) required to estimate the mean radius of gyration of an equilateral $n$-gon in $\R^3$ in the $\PolQuotProb$ measure with $99\%$ confidence to within $0.1\%$ of the sample mean. The experiments were conducted on an Apple~M1~Max with 8 parallel CPU threads.}
\label{fig:perf}
\end{figure}


\section{Conclusion}

We have given a new algorithm for sampling configurations of closed polygons with arbitrary prescribed edgelengths in any dimension. Our method can sample the same symplectic volume on $\PolQuot$ for $d=3$ as the Progressive Action-Angle method~\cite{CSS}. However, it is much more general, allowing us to construct samples with arbitrary edgelengths, in other dimensions, and with respect to the measure $\PolProb$ as well as the measure $\PolQuotProb$ on the quotient. We provide an open-source implementation of our algorithm:
see \cite{Cantarella_CoBarS_-_Conformal} for a parallel, header-only implementation in \emph{C++}; and see \cite{Cantarella_CoBarSLink_-_A} for a \emph{Mathematica} interface.  Our algorithm runs in $O(n)$ time as opposed to the $O(n^2)$ complexity of the Progressive Action Angle method. Though time-per-sample can be misleading for reweighted samplers, our performance comparisons show that the new sampler is faster at estimating means to fixed confidence intervals. 

We now suggest some avenues for further investigation. Many authors have studied the homology of the space of planar polygons with fixed edgelengths~\cite{Farber2007,Kamiyama2015}. It would be interesting to combine our sampling algorithm with tools from topological data analysis to find cohomology groups for polygon spaces computationally. This could give new insight into the remaining open questions regarding the cohomology rings of polygon spaces in higher dimensions as well, as in~\cite{FHS2011}. 

The volumes of polygon spaces in any dimension can (in principle) be computed explicitly by integrating the reweighting factors of~\eqref{eq:final sampling weights} and~\eqref{eq:final quotient sampling weights} over $\Arm$. 
Very little appears to be known about these volumes. In~2d, Kamiyama~\cite{Kamiyama2012} has given computational results on the volume of the space of polygons with one long edge for $n=4,5,6$. In~3d, Khoi~\cite{Khoi2005} used Witten's formula to compute symplectic volumes for all polygon spaces. Khoi proved that among all $n$-gons with the same total edgelength in $\R^3$, the equilateral $n$-gons are the most flexible in the sense that their configuration space has the largest symplectic volume. These methods are deeply symplectic and so tied to the three dimensional case. But it is very natural to ask: are the equilateral $n$-gons the most flexible in $\R^d$? Our algorithm provides a powerful toolkit for this kind of investigation.

\bibliographystyle{abbrvhref}

\bibliography{papers,dlmf}

\end{document}